\documentclass[11pt,reqno]{amsart}
\usepackage[shortlabels]{enumitem}
\usepackage{amssymb,commath,amscd}
\usepackage{txfonts}
\usepackage{amsthm}
\usepackage{mathrsfs}
\usepackage{thmtools}
\usepackage{amsmath}
\usepackage{
}
\usepackage[colorlinks,linkcolor=black,anchorcolor=blue,citecolor=green]{hyperref}
\usepackage[all]{xy}

\declaretheorem[numberwithin=section]{theorem}
\declaretheorem[sibling=theorem]{lemma}
\declaretheorem[sibling=theorem]{corollary}
\declaretheorem[sibling=theorem]{proposition}

\declaretheorem[sibling=theorem, style=remark]{remark}

\declaretheorem[sibling=theorem, style=remark]{definition}
\numberwithin{equation}{section}
\allowdisplaybreaks

\newcommand{\injrad}{\operatorname{inj.rad}}
\newcommand{\conjrad}{\operatorname{conj.rad}}
\newcommand{\Ric}{\operatorname{Ric}}
\newcommand{\diam}{\operatorname{diam}}
\newcommand{\vol}{\operatorname{vol}}

\newcommand{\op}[1]{\operatorname{#1}}

\begin{document}
	\title[Convergence of Ricci-limit spaces] {Convergence of Ricci-limit spaces under bounded Ricci curvature and local covering geometry I}
	
	
	\author{Zuohai Jiang}
	\address{Beijing international center for mathematical research, peking university, Beijing, China}
	\email{jiangzuohai08@pku.edu.cn}
	
	
	\author{Lingling Kong}
	\address{School of Mathematical and statistical,  Northeast Normal Universiy, Changchun China}
	\curraddr{}
	\email{kongll111@nenu.edu.cn}
	\thanks{The second author would like to thank Capital Normal University for
		a warm hospitality during his visit}
	\author{Shicheng Xu}
	\address{School of Mathematical Sciences, Capital Normal Universiy, Beijing China}
	\address{Academy for Multidisciplinary Studies, Capital Normal University, Beijing  China}
	\curraddr{}
	\email{shichengxu@gmail.com}

	
	\date{\today}
	\begin{abstract}
		We extend Cheeger-Gromov's and Anderson's convergence theorems to regular limit spaces of manifolds with bounded Ricci curvature and local covering geometry, by establishing the $C^{1,\alpha}$-regularities that are the best one may expect on those Ricci-limit spaces. As an application we prove an optimal generalization of Fukaya's fibration theorem on collapsed manifolds with bounded Ricci curvature, which also improves the original version to $C^{1,\alpha}$ limit spaces.
	\end{abstract}
	
	\maketitle
	
	\setcounter{section}{-1}
	\section{Introduction}\label{section-0}
	
One of the fundamental tools in the study of geometry and topology of manifolds with curvature bound is the
Cheeger-Gromov's convergence theorem \cite{Cheeger1970,GLP1981}, which implies that
the space $\mathcal{M}^n(v,D)$ of all closed Riemannian $n$-manifolds of sectional curvature $|\sec|\leq 1$, volume $\geq v>0$ and diameter $\leq D<+\infty$ is precompact in the $C^{1,\alpha}$-topology (cf. also \cite{Peters1987,GW1988,Kasue1989}). This $C^{1,\alpha}$-convergence theorem had been generalized by Anderson \cite{Anderson1990} (cf. \cite{Gao1990}) and Anderson-Cheeger \cite{AnCh1992} to Riemannian manifolds with (lower) bounded Ricci curvature. That is, the space $\mathcal{M}_{|Ric|}^n(i_0,D)$ (resp. $\mathcal{M}_{Ric}^n(i_0,D)$) consisting of all closed Riemannian $n$-manifolds of Ricci curvature $|\Ric|\leq n-1$ (resp. $\Ric\ge -(n-1)$), injectivity radius $\geq i_0>0$ and diameter $\leq D<+\infty$ is precompact in the $C^{1,\alpha}$-topology (resp. $C^{0,\alpha}$-topology). A weaker replacement of the injectivity radius is \emph{$(\delta,\rho)$-Reifenberg condition} \cite{CC1997I} (cf. \cite{HKRX2020}) for $0\le \delta\leq \delta(n)$,
i.e., around each point $x$ in such a manifold the Gromov-Hausdorff distance
\begin{equation}\label{def-Reifenberg-*}
d_{GH}\left(B_{r}(x), B_{r}^{n}(0)\right)\le \delta \cdot r, \qquad \text{for all $0<r\le \rho$},
\end{equation} where $B_{r}^{n}(0)$ is the $r$-ball at the origin in $\mathbb R^n$, $\delta(n)$ is a constant depending only on $n$, and $\rho$ is a small positive constant. A point satisfying \eqref{def-Reifenberg-*} is called a $(\delta,\rho)$-Reifenberg point.
By Anderson \cite{Anderson1990} and Colding \cite{Colding1997}, it was well known that Anderson's convergence theorem still holds after replacing the lower bound of injectivity radius by the $(\delta,\rho)$-Reifenberg condition.
\par	
By the convergence theorems above, a limit space of manifolds in $\mathcal{M}^n_{Ric}(i_0,D)$ (resp. $\mathcal{M}^n_{|Ric|}(i_0,D)$) in the Gromov-Hausdorff topology is always a smooth manifold with a $C^{\alpha}$-Riemannian metric (resp. $C^{1,\alpha}$-Riemannian metric).
However, without a positive injectivity radius bound a Ricci limit space may contain non-Euclidean points. A collapsed Ricci limit space may even have non integer Hausdorff dimension, see \cite{PanWei2021}.
The geometric properties of Ricci limit spaces had played a fundamental role in solving some open problems and conjectures on manifolds with (lower) Ricci curvature bound; see for example \cite{CC1996}, \cite{CC1997I}, \cite{KW}.
\par	
A Riemannian manifold $(M,g)$ with normalized curvature bound is called $\epsilon$-collapsed, if the volume of any unit ball $B_{1}(x)$ at $x\in M$ is less than $\epsilon$. As the counterpart of Cheeger-Gromov's convergence theorem, $\epsilon$-collapsed Riemannian manifolds with bounded sectional curvature had been extensively studied by Cheeger-Fukaya-Gromov \cite{CFG1992} (cf. Cheeger-Gromov \cite{CheegerGromov1990, CheegerGromov1990II}, and
Fukaya \cite{Fukaya1987, Fukaya1988, Fukaya1989}).
In contrast, the geometry and topology of collapsed manifolds with bounded Ricci curvature are much more complicated and rarely known at present. Many progresses on the collapsed manifolds with lower bounded Ricci curvature are achieved recently under some additional geometric assumptions, such as bounded local covering geometry, see \cite{Rong2018}, \cite{CRX2019,CRX2017}, \cite{NaberZhang2016}, \cite{HKRX2020}, \cite{Huang2020}, \cite{HW2020-1,HW2020-2}, etc. and the survey paper \cite{HRW2020}.

According to \cite{HKRX2020} (cf. \cite{CGT1982,Cheeger-Gromov1985}), a complete Riemannian $n$-manifold $(M,g)$ with $|\Ric_M|\leq n-1$ is said to have \emph{$(r,v)$-local covering geometry}, if for any $x\in M$, the local rewinding volume of $B_r(x)$, $\op{Vol}(B_r(\tilde x))\ge vr^n>0$, where $\tilde x$ is a preimage point of $x$ in the (incomplete) Riemannian universal cover $\pi: (\widetilde{B_{r}(x)},\tilde x)\to (B_{r}(x),x)$.  Moreover, $(M,g)$ is said to have \emph{$(\delta,r)$-Reifenberg local covering geometry}, if for any $x\in M$, $\tilde x$ is a $(\delta,r)$-Reifenberg point. By Cheeger-Fukaya-Gromov \cite[Theorem 1.3]{CFG1992}, for any $\delta>0$, there is $r(n,\delta)>0$ such that any collapsed complete manifold with sectional curvature bound $|\sec|\le 1$ admits $(\delta,r(n,\delta))$-Reifenberg local covering geometry.

In this paper, we prove the $C^{1,\alpha}$-regularity of those regular limit spaces under bounded Ricci curvature and local covering geometry, which naturally extends the above Cheeger-Gromov's and Anderson's convergence theorems.

Let $\mathcal{X}_{n,r,v}^m(\delta,\rho)$ be the set consisting of all compact Ricci-limit spaces of Riemannian $n$-manifolds with $|\Ric_M|\leq n-1$ and $(r,v)$-local covering geometry, such that the dimension of $X$ is $m$ in the sense of Colding-Naber \cite{ColdingNaber2011}, and any point $x$ in $X\in \mathcal{X}_{n,r,v}^{m}(\delta,\rho)$ is $(\delta,\rho)$-Reifenberg.

\begin{theorem}[bounded local covering geometry]\label{thm-bounded-local-covering-geometry}
	Given $r,v,\rho>0$ and positive integers $m\leq n$, there are constants $\delta=\delta(n),T(n,r,v,\rho)>0$ such that the followings hold.
	\numberwithin{enumi}{theorem}
	\begin{enumerate}[leftmargin=0pt,itemindent=*]
		\item \label{full-rank1} Any $X\in\mathcal{X}_{n,r,v}^m(\delta,\rho)$ is a $C^{1,\alpha}$-Riemannian manifold $(X,h)$ with a positive $C^{1,\alpha}$-harmonic radius $\ge r_0(n,r,v,\rho,\alpha)>0$ for any $0<\alpha<1$.  And for each $0<\epsilon\leq T(n,r,v,\rho)$, there is a nearby Riemannian metric $h(\epsilon)$ on $(X,h)$ such that
		\begin{enumerate}[(a)]
			\item\label{full-rank1a} $\|h(\epsilon)-h\|_{C^{1,\alpha}}\le \Psi(\epsilon\,|\,n,r,v,\rho,\alpha)$, where $\Psi(\epsilon\,|\,n,r,v,\rho,\alpha)\to 0$ as $\epsilon\to 0$ with other parameters $n,\dots, \alpha$ fixed,
			\item\label{full-rank1b} the sectional curvature of $h(\epsilon)$ satisfies $|\sec_{h(\epsilon)}|\le C(n,r,v,\rho)\epsilon^{-\frac{1}{2}}$,
			\item\label{full-rank1c} the $k^{\text{th}}$-derivative of curvature tensor $\left|\nabla^{k}\operatorname{Rm}(h(\epsilon))\right|_{h(\epsilon)}\leq C(n,r,v,\rho,k,\epsilon),$
		\end{enumerate}
		where the constants $r_0(n,r,\dots)$ and $C(n,r,\dots)$ depend only on the given parameters.
		\item \label{full-rank2}
		The subset $\mathcal{X}_{n,r,v}^m(\delta,\rho,D)=\{X\in \mathcal{X}_{n,r,v}^m(\delta,\rho) : \diam(X)\le D\}$ is compact in the $C^{1,\alpha}$-topology. In particular, $\mathcal{X}_{n,r,v}^m(\delta,\rho,D)$ contains only finitely many diffeomorphism types.
	\end{enumerate}
\end{theorem}

Recently Naber and Zhang \cite{NaberZhang2016} proved the $\epsilon$-regularity theorem for locally full-rank collapsed Riemannian manifolds with (lower) bounded Ricci curvature, which says that there are uniform constants $\epsilon=\epsilon(n),r(n,\alpha)>0$ such that for any complete Riemannan $n$-manifold $M$ with $|\Ric_{M}|\le n-1$ and $0<\rho\le 1$, if an open ball $B_{\rho}(x)$ is $\epsilon \rho$-Gromov-Hausdorff close to a $\rho$-ball in a lower dimensional Euclidean space $\mathbb R^m$, then the subgroup $\Gamma_{\epsilon, \rho}$ generated by loops at $x$ of length $< 2\epsilon\rho$ in the fundamental group $\pi_1(B_\rho(x),x)$ has the full rank $n-m$ if and only if the preimages of $x$ in the universal cover of $B_{\rho/2}(x)$ admit a uniform $C^{1,\alpha}$-harmonic radius $\ge r(n,\alpha)\rho>0$. Recall that by the generalized Margulis lemma \cite{KW}, the local fundamental group $\Gamma_{\epsilon, \rho}$
contains a nilpotent subgroup $N$ of finite index $\le \omega(n)$, where the rank of $\Gamma_{\epsilon, \rho}$ is defined to be that of $N$, which is no more than $n-m$ (see e.g., \cite[Theorem 2.27]{NaberZhang2016}).

Let $\mathcal{Y}_n^{m}(\delta,\rho)$ be the set consisting of all compact Ricci-limit spaces of locally full-rank collapsed manifolds with $|\operatorname{Ric}|\le n-1$, such that the dimension of $X\in\mathcal{Y}^{m}_n(\delta,\rho)$ is $m$, and any point in $X$ is $(\delta,\rho)$-Reifenberg. By Naber-Zhang's $\epsilon$-regularity, $\mathcal Y_n^m(\delta,\rho)\subset \mathcal{X}_{n,r,v}^{m}(\delta,\rho)$ for some $r=r(n,\rho)$ and $v>0$.

\begin{corollary}[locally full-rank collapsed limit spaces]\label{cor-full-rank-collapsed-Ricci-limit-space-*}
	 Let $\delta=\delta(n)>0$ be the constant in Theorem \ref{thm-bounded-local-covering-geometry}.
	 The conclusions in Theorem \ref{thm-bounded-local-covering-geometry} hold for $\mathcal{Y}_n^{m}(\delta,\rho)$, where the dependence of constants $T, C$ and $r_0$ on $n,r,v,\rho$ can be simplified to $n$ and $\rho$ only.
\end{corollary}
\par

Since all closed manifolds in $\mathcal{M}_{|Ric|}^m(i_0,D)$ are contained in $\mathcal{Y}^m_n(\delta(n),\rho,D)$ for some $\rho=\rho(n,i_0)$,
Theorem \ref{thm-bounded-local-covering-geometry} generalizes Anderson's $C^{1,\alpha}$-convergence theorem.

Let us recall that in general, a positive lower volume bound for manifolds with bounded Ricci curvature and diameter is weaker than
a positive injectivity radius bound, though they are equivalent under sectional curvature bound.
Similar to the Anderson's \cite{Anderson1990} and Anderson-Cheeger's \cite{AnCh1992} convergence theorems, Theorem \ref{thm-bounded-local-covering-geometry} fails after loosing the $(\delta,\rho)$-Reifenberg condition to a positive volume lower bound on limit spaces; it shares the same counterexamples as those for Anderson's convergence (see \cite{Anderson1990}).

Next, we will generalize the Cheeger-Gromov's convergence to limit spaces of manifolds with bounded Ricci curvature and  Reifenberg-bounded local covering geometry, where the Reifenberg condition (resp. positive injectivity radius) on the limit spaces is replaced by a positive lower bound of the $m$-Hausdorff measure (resp. the volume).

Let $\mathcal{Z}_{n,\delta,r}^m(\tau)$ be the set consisting of all compact Ricci-limit spaces of Riemannian $n$-manifolds with $|\Ric_M|\leq n-1$ and $(\delta,r)$-Reifenberg local covering geometry, such that each element $X\in \mathcal{Z}_{n,\delta,r}^{m}(\tau)$ is $\tau$-almost regular in the sense that for any $x\in X$, any tangent cone $(T_xX,x)$ at $x$ is $\tau$-close to $(\mathbb{R}^{m},0)$ in the pointed Gromov-Hausdorff topolgy. 

\begin{theorem}[Reifenberg-bounded local covering geometry]\label{thm-Reifenberg-local-covering-geometry}
	Let $\delta=\delta(n)>0$ be the constant in Theorem \ref{thm-bounded-local-covering-geometry}.
	\begin{enumerate}[leftmargin=0pt,itemindent=*]
		\item\label{local-covering-geometry-1} Any $X\in \mathcal{Z}_{n,\delta,r}^{m}(\delta)$ is a $C^{1,\alpha}$-Riemannian manifold $(X,h)$ such that for any point $x\in X$, $\op{Vol}(B_1(x))\ge w>0$ implies that the $C^{1,\alpha}$-harmonic radius at $x$ is no less than $r_h(n,r,w,\alpha)>0$ for any $0<\alpha<1$.
		\item\label{local-covering-geometry-2} Let $\mathcal{Z}_{n,\delta,r}^{m}(\delta,w,D)=\{X\in\mathcal{Z}_{n,\delta,r}^{m}(\delta): \op{diam}(X)\le D, \op{Vol}(X)\ge w>0\}$. There is $\rho=\rho(n,r,w,D)>0$ such that $\mathcal{Z}_{n,\delta,r}^{m}(\delta,w,D)\subset \mathcal Y_{n}^m(\delta,\rho,D)$.
		
		In particular,
		the conclusions in Theorem \ref{thm-bounded-local-covering-geometry} hold for $\mathcal{Z}_{n,\delta,r}^{m}(\delta,w,D)$.
	\end{enumerate}
\end{theorem}

 Theorem \ref{thm-Reifenberg-local-covering-geometry}  generalizes the local estimate \cite{CGT1982}, \cite{ChLiYau} on the injectivity radius of manifolds with bounded sectional curvature to harmonic radius of limit spaces under bounded Ricci curvature and Reifenberg  local covering geometry.

We point out that the Theorem \ref{thm-Reifenberg-local-covering-geometry} fails for regular Ricci-limit spaces under bounded Ricci curvature and $(r,v)$-local covering geometry in the sense of rewinding volume, since it would contains all non-collapsed Ricci-flat manifolds.
	
\begin{remark}\label{rem-main-1}
	For those regular limit spaces of collapsed manifolds with two-sided bounded sectional curvature, the $C^{1,\alpha}$-regularity and compactness in Theorems \ref{thm-bounded-local-covering-geometry}-\ref{thm-Reifenberg-local-covering-geometry} are well-known to experts and can be easily derived by \cite{Fukaya1988}. In fact, they are direct corollaries of Cheeger-Gromov's convergence theorem, because those limit spaces of Hausdorff dimension $m$ can be smoothed to metrics whose sectional curvature is bounded two-sided uniformly by $C(n,w,D)$, where $\diam(X)\le D$, and $m$-Hausdorff measure $H^m(X)\ge w>0$; for details see Section 2, and also compare \cite[Theorem 0.9]{Fukaya1988} and \cite[Corollary 0.11]{Fukaya1988}, which states $X$ is a smooth manifold with a continuous metric tensor $h$ inducing a $C^{1,\alpha}$ distance function. 	
	
	However, the smoothed metrics on limit spaces of manifolds with bounded Ricci curvature generally admit no uniformly bounded sectional curvature.
	In order to derive the $C^{1,\alpha}$-convergence,
	one has to construct $C^{1,\alpha}$-harmonic coordinates directly on a limit space. This is the new ingredient in Theorem \ref{thm-bounded-local-covering-geometry}.
	
	For those non-collapsed $n$-manifolds with $|\Ric_M|\leq n-1$ and $(\delta,r)$-Reifenberg local covering geometry, the $C^{1,\alpha}$-precompactness of non-collapsed has been proved in \cite[Theorem E]{CRX2017}.
	
	It should be pointed out that, though a limit space $X$ in Theorems  \ref{thm-bounded-local-covering-geometry}-\ref{thm-Reifenberg-local-covering-geometry} admits the sythetic $CD(-(n-1), n)$ curvature condition (cf. \cite{GKMS2018}) or Bakry-\'Emery-Ricci curvature lower bound in a generalized sense by \cite{Lott2003}, their weighted measures cannot be used to detect how much $X$ is collapsed as in Theorem \ref{thm-Reifenberg-local-covering-geometry}. On the other hand, for the original $m$-Hausdorff measure on $X$, we do not know whether the volume comparison is satisfied; compare \cite{Pro-Wilhelm2014}.
\end{remark}
\begin{remark}\label{rem-main-2}
	It is well known by \cite{PWY1999} that once $X$ admits a positive $C^{1,\alpha}$-harmonic radius $r>0$, it can be smoothed to a new metric $h_\epsilon$, whose sectional curvature $|\op{sec}_{h_\epsilon}|\le C(r,m,\epsilon)$. The smoothed metric $h(\epsilon)$ in Theorem \ref{thm-bounded-local-covering-geometry} has a better order (\ref{full-rank1}.b-c), which arises from the Ricci flow solutions $g(\epsilon)$ on such manifolds, where $h(\epsilon)$ is their limit metric under Gromov-Hausdorff topology.
	
	Indeed, by \cite{HKRX2020} (see also Lemma \ref{lem-rewinding-reifenberg} below), for any $  X\in \mathcal{X}_{n,r,v}^m(\delta,\rho),$ the manifolds in its converging sequence with bounded Ricci curvature and $(r,v)$-local covering geometry  satisfy $(\delta,r')$-Reifenberg local covering geometry for some $r'=r'(n,r,v,\rho)>0$. Then by Dai-Wei-Ye \cite{DWY1996} (see also Theorem \ref{thm-smoothing-ricci-flow} below), the solution $g(t)$ of Ricci flow equation with initial value $g$ exists in $(0,T(n,r')]$ for some constant $T(n,r')>0$, and satisfies (\ref{thm-bounded-local-covering-geometry}.1.a-c) on the local universal cover for $0<\epsilon=t\le T(n,r')$. Hence any limit space in $\mathcal{X}_{n,r,v}^{m}(\delta,r')$ admits a nearby metric $h(t)$ that locally is a quotient orbit space of manifolds with bounded sectional curvature $C(n,r')t^{-1/2}$.
	
	By elementary facts on Riemannian submersions (e.g., see the proof of Lemma \ref{almost-C1-harmonic-coordinate-chart-on-X} below), the curvature condition on $g(t)$ naturally passes to the quotient metric $h(t)$ in a harmonic coordinate chart (see the diagram \eqref{eqGH-local-covers} below). Note that, though the lower curvature bound can be always passed to $h(t)$ by the O'Neill's formula, the radius of harmonic coordinates on the quotient is crucial for the upper curvature bound.
	There are limit spaces, e.g. \cite[\S 1-e, Example 1.13]{Fukaya1986}, whose sectional curvature blows up as the volume goes to zero.
\end{remark}

\par
We now give an application of Theorem \ref{thm-bounded-local-covering-geometry}. As a parametrized version of Gromov's almost flat manifold theorem (\cite{Gromov1978}, \cite{Ruh1982}), Fukaya \cite{Fukaya1987} constructed a bundle structure whose fibers absorb all collapsing directions on a manifold $M$ that is Gromov-Haussdorff close to a lower dimensional manifold under bounded sectional curvature.

\begin{theorem}[Fukaya's fibration theorem \cite{Fukaya1987,CFG1992}]\label{thm-fukaya}
	Given constants $n\ge2$, $1\ge i_0>0$, there are constants $\epsilon(n)>0$ and $C(n)>0$ such that the following holds.
	\par		
	Let $(M,g)$ and $(N,h)$ be a closed Riemannian $n$-manifold and $m ( \leq n )$-manifold respectively, whose sectional curvature and injectivity radius satisfy
	\begin{equation*}
	|\sec_{(M,g)}|\leq 1,\quad |\sec_{(N,h)}|\leq 1,\quad \injrad(N,h)\geq i_0.
	\end{equation*}
	If $d_{GH}(M,N)\le \epsilon\cdot i_0$ with $\epsilon< \epsilon(n)$, then there is a $C^\infty$-smooth fibration $f:M\to N$ such that
	\begin{enumerate}\numberwithin{enumi}{theorem}
		\item\label{thm-fukaya-1} $f$ is a $\varkappa(\epsilon\,|\,n)$-almost Riemannian submersion, i.e., for any vector $\xi$ perpendicular to an $f$-fiber,
		$e^{-\varkappa(\epsilon\,|\,n)}|\xi|_{g}\le |d f(\xi)|_h\leq e^{\varkappa(\epsilon\,|\,n)}|\xi|_{g},$
		where after fixing $n$, $\varkappa(\epsilon\,|\,n)\to 0$ as $\epsilon\to 0$.
		\item\label{thm-fukaya-2} The intrinsic diameter of any $f$-fiber $F_{x}=f^{-1}(x)$ over $x\in N$ satisfies $\op{diam}_{g}(F_x) \le C(n)\cdot d_{GH}(M,N).$
		\item\label{thm-fukaya-3} The second fundamental form is bounded by
		$\left|\nabla^2f\right|\leq C(n)i_0^{-1}.$
		\item\label{thm-fukaya-4} $F_x$ is diffeomorphic to an infra-nilmanifold.
	\end{enumerate}
\end{theorem}
\begin{remark}\label{rem-fukaya-fibration}
	The formulation of Theorem \ref{thm-fukaya} is similar to \cite[Theorem 2.6]{CFG1992}, where the estimates are better than its original versions \cite{Fukaya1987,Fukaya1989}, but depend on the higher regularities of $g$ and $h$, called $A$-regular in \cite{CFG1992}, i.e., the curvature tensor satisfies $|\nabla^i \operatorname{Rm}|\le A_i$ for all integer $i\ge 0$. It is well-known that the dependence of  $\{A_i\}_{i\ge 1}$ in (\ref{thm-fukaya}.1-3) can be removed in several ways, e.g. see \cite{Rong2019} for a simple proof, and also Theorem \ref{thm-fibration-of-C1-manifolds} below.
\end{remark}

The last main result in this paper is an optimal generalization of Fukaya's fibration theorem on collapsed manifolds under bounded Ricci curvature. Let $\delta(n)>0$ be the constant in Theorem \ref{thm-bounded-local-covering-geometry}.
\begin{theorem}\label{thm-fibration-of-C1-manifolds}
	Given $\rho>0$ and positive integers $n, m(\leq n)$, there exist constants $\epsilon(n), C(n)$ such that the following holds.
	\par		
	Let $(M,g)$ be a closed Riemannian $n$-manifold
	with $|\Ric_M|\leq n-1$ and $(\delta(n),\rho)$-Reifenberg local covering geometry, and $(X,h)\in \mathcal{Y}_{n}^{m}(\delta(n),\rho)$ for $0<\rho\le 1$.
	If $d_{GH}(M,X)\le \epsilon \cdot \rho$ with $\epsilon<\epsilon(n)$, then there is a $C^\infty$-smooth fibration $f:M\to X$ that satisfies \eqref{thm-fukaya-1}-\eqref{thm-fukaya-4} after replacing $i_0$ with $\rho$.
\end{theorem}

Theorem \ref{thm-fibration-of-C1-manifolds} also holds for closed Riemannian $n$-manifolds with $|\Ric_M|\leq n-1$ and $(r,v)$-local covering geometry, because by Lemma \ref{lem-rewinding-reifenberg} it admits a uniform Reifenberg-bounded local covering geometry.

The existence of a fibration that is a $\varkappa(\epsilon\,|\,n)$-Gromov-Hausdorff approximation satisfying \eqref{thm-fukaya-1} and \eqref{thm-fukaya-4} was already well known (cf. \cite[Proposition 6.6]{NaberZhang2016}). In fact, after removing the upper Ricci curvature bound in Theorem \ref{thm-fibration-of-C1-manifolds}, a smooth fibration that is $\varkappa(\epsilon\,|\,n)$-Gromov-Hausdorff approximation was constructed in both \cite{HKRX2020} (by smoothing methods based on Perelman's pseudo-locality \cite{Perelman2002} for the Ricci flow) and \cite{Huang2020} (by gluing $\delta$-splitting maps together via center of mass), where the uniform regularity is H\"older. Such fibrations are also constructed between Alexandrov spaces \cite{Fujioka2021} recently.

What is new for the fibration in Theorem \ref{thm-fibration-of-C1-manifolds} is that, it provides the best possible regularity \eqref{thm-fukaya-3}, which is first known in the literature even for the case that $(X,h)$ is an Euclidean $1$-ball (cf. \cite[Proposition 6.6]{NaberZhang2016}).

\begin{remark}
	The fibration in Theorem \ref{thm-fibration-of-C1-manifolds} is constructed  via gluing the locally defined Cheeger-Colding's $\delta$-splitting maps together. However, the optimal regularities are not direct consequences of neither those smoothing methods (e.g., \cite{DWY1996} \cite{PWY1999}), nor  the Cheeger-Colding's $L^2$-estimates \cite{CC1996,CC1997I} on the $\delta$-splitting map. The subtle point is the balance between \eqref{thm-fukaya-1} and \eqref{thm-fukaya-3}. For example, if a fiber bundle $f_t$ is constructed with respect to a smoothed metric $g(t)$ by the earlier known methods, then $f_t:(M,g)\to (X,h)$ is a $\varkappa(\epsilon,t\,|\,n)$-almost Riemannian submersion depending also on $t$. In order to derive \eqref{thm-fukaya-1}, $t$ has to approach $0$, while $\op{sec}_{g(t)}$ and hence $\nabla^2f_t$ generally blows up as $t\to 0$. Similar issue also happens in applying Cheeger-Colding's $L^2$-estimates. Instead, we apply the $C^{2,\alpha}$-compactness of harmonic coordinate charts on the local covers, which is crucial in deriving the optimal regularities for Theorem \ref{thm-fibration-of-C1-manifolds}. For details, see Remarks \ref{rem-local-almost-submersion} and \ref{rem-fibration-smooth} below.
	
\end{remark}

\begin{remark}\label{rem-fibration-difference}
Compared with those earlier versions of the fibration Theorem in \cite{Fukaya1989}, \cite{CFG1992}, \cite{NaberZhang2016}, \cite{HKRX2020}, \cite{Huang2020}, etc., another improvement here is that, $(X,h)$ has only $C^{1,\alpha}$-regularity that may even admit no standard exponential map (see \cite{Hartman-Wintner1951}, cf. \cite{CalabiHartman1970}).  We will apply the center of mass technique with respect to a smoothed nearby metric  $h(t_0)$ offered by Theorem \ref{thm-bounded-local-covering-geometry}, which admits a convex radius depends on $t_0$, such that \eqref{thm-fukaya-1}-\eqref{thm-fukaya-4} are proved for the original metric $h$ with fixed $t_0$.
\end{remark}

\par
At the core of Theorem \ref{thm-bounded-local-covering-geometry} is the proof of the existence of harmonic coordinates, i.e., the charts for which the coordinate functions are harmonic functions on balls of a uniform size (depending only on the constants given), and uniform $C^{1,\alpha}$-norm estimates of the metric tensor in these coordinates \cite{Cheeger1970,JostKarcher1982}. The main ingredients are as follows.

Let $(M_i,g_i)$ be a sequence of Riemannian $n$-manifolds with $|\op{Ric}_{M_i}|\le n-1$ and $(r,v)$-local covering geometry that converges to $X\in \mathcal X_{n,r,v}^m(\delta,\rho)$ in the Gromov-Hausdorff topology. By Lemma \ref{lem-rewinding-reifenberg} and \cite{Anderson1990}, for $\delta=\delta(n)>0$, the $C^{1,\alpha}$-harmonic radius of $\tilde x_i$ in the universal cover of $B_{r}(x_i)$ admits a uniform lower bound $r_0(n,r,v,\rho)>0$. For simplicity we assume $r=\rho$.

According to the precompactness principle for domains with boundary \cite{Xu2021} (see Theorem \ref{thm-precompactness-a} below), there is a normal cover $\widehat{U}_i$ of $B_{\frac{\rho}{2}}(x_i, g_i)$ equipped with its length metric such that by passing to a subsequence, it converges equivariantly in the pointed Gromov-Hausdorff topology:
\begin{equation}\label{eqGH-local-covers}
\begin{CD}
(\widehat{U}_i, \hat{x}_i, \Gamma_i) @>GH>i\to \infty> (Y, \hat x_\infty, G)\\
@V\pi_iVV @V\pi_\infty VV\\
(B_{\frac{\rho}{2}}(x_i, g_i),x_i) @>GH>i\to \infty> Y/G,
\end{CD}
\end{equation}
where $\Gamma_i$ is the deck-transformation of $\Gamma_{\frac{1}{2},\rho}(x_i)$, $G$ is the limit group of $\Gamma_i$, and the quotient $Y/G$ is locally isometric to the limit ball $B_{\frac{\rho}{2}}(x_\infty)$ on $X$.
Moreover, by the definition of $\widehat{U}_i$ (see Remark \ref{rem-Reifenberg-normal-cover} below), it still admits a uniform $C^{1,\alpha}$-harmonic radius lower bound. Hence $Y$ is a $C^{1,\alpha}$-Riemannian manifold.

In order to present the idea shortly, we first assume that $Y$ is a smooth Riemannian manifold.
Since the tangent cone of $Y/G$ is the quotient space of $\mathbb R^n$, which is either isometric to $\mathbb R^m$ or definitely away from $\mathbb R^m$, the $(\delta,\rho)$-Reifenberg condition implies that $Y/G$ is regular. By the standard theory of isometric actions on Riemannian manifolds (e.g. \cite[\S 1]{Grove2000}), $Y/G$ is also a Riemannian manifold. Then we construct a harmonic coordinate on $X$ in the following two steps.

Step 1. Following Cheeger-Colding \cite{CC1996,CC1997I}, we construct a harmonic $\delta$-splitting map $\varphi_i:B_{\delta^{-1/4}\rho}(x_i, \delta^{-1}g_i)\to \mathbb R^m$. We lift $\varphi_i$ to a harmonic $\delta$-splitting map $\hat \varphi_i=\varphi_i\circ \pi_i$ on $\hat U_i$, and then by appending other harmonic functions, we complete it to a harmonic coordinate chart $(\hat \varphi_i, \hat \psi_i):B_{\delta^{-1/4}\rho}(\hat x_i, \delta^{-1}\hat g_i, \widehat{U}_i)\to \mathbb R^n$.

Step 2. By taking limit of $(\hat \varphi_i, \hat \psi_i)$, we get a harmonic coordinate chart $(\hat \varphi_\infty, \hat \psi_\infty)$ on $Y$ such that $\hat \varphi_\infty$ takes the same value on each $G$-orbit. Hence $\hat \varphi_\infty$ descends to a smooth map $\varphi_\infty$ on $X$. For simplicity such harmonic coordinate chart is called to be \emph{adapted} for a submersion $\pi$, i.e., each $y^j$ $(j=1,\dots, m)$ takes the same value along every $\pi$-fiber.

By the $C^{1,\alpha}$-precompactness on $\widehat{U}_i$ via harmonic coordinates, $(\hat \varphi_\infty, \hat \psi_\infty)$ admits a small Hessian up to a definite rescaling on the metric. By the technical result below, $\varphi_\infty$ gives rise to a harmonic coordinate chart of definite size on $X$.

\begin{theorem}\label{tech-theorem-*}
	Given any $r>0, 0<\alpha<1, 0<Q\le 10^{-2}$ and integers $n,m(\leq n)$, there is a constant $\tau(n,r,\alpha,Q)>0$ such that the following holds.
	\par	
	Let $\pi:(Y,g)\to (X,h)$ be a Riemannian submersion from a Riemannian $n$-manifold (may not complete) to  a Riemannian $m$-manifold. Suppose that there is an adapted $C^{1,\alpha}$-harmonic coordinate chart $(y^1,\dots, y^n): B_r(p)\to \mathbb R^n$ at $p\in Y$ with $(\alpha,Q)$-$C^{1,\alpha}$-control (see \eqref{def-hr-1}-\eqref{def-hr-2} below).
	If the Hessian of each adapted coordinate function satisfies
	\begin{equation}\label{Hessian-control-for-yj-*}
	\|\operatorname{Hess}y^j\|_{C^{0,\alpha}(B_{r}(p))}\leq \tau(n,r,\alpha,Q), \ \ j=1,\dots,m,
	\end{equation}
	where the $C^{0,\alpha}$-norm is taken in the coordinates $(y^{1},\dots,y^n)$,
	then there is a $C^{1,\alpha}$-harmonic coordinate chart $(x^1,\dots, x^m): B_{r/2}(\bar{p})\to \mathbb R^m$ at $\bar{p}=\pi(p)$ with $(\alpha,2Q)$-$C^{1,\alpha}$-control.
\end{theorem}

Note that the limit space $Y$ of the normal covers of balls in \eqref{eqGH-local-covers} is only a $C^{1,\alpha}$-Riemannian manifold, and in general, a $C^{1,\alpha}$-Riemannian manifold may even admit no standard exponential map (see \cite{Hartman-Wintner1951}, cf. \cite{CalabiHartman1970}). In order to guarantee the arguments above, we will show that $\pi_\infty$ is still a smooth submersion between $C^{1,\alpha}$-Riemannian manifolds. This can be seen by applying the Ricci flow on $(M_i,g_i)$ to obtain smooth limits $Y_t$ and $X_t$, which share the same limit group action in \eqref{eqGH-local-covers}; see Proposition \ref{prop-C1alpha-regularity} below. Thus a harmonic coordinate chart on $X$ can be constructed as above.

The organization of this paper is as following. In section
\ref{preliminaries}, we will supply some notations and preliminary facts that will be used later. In section 2 we give a simple proof of the $C^{1,\alpha}$-compactness for limit spaces under bounded sectional curvature.
Section 3 is devoted to the proof of Theorem \ref{tech-theorem-*}.
In section 4, we shall prove that each element $X\in\mathcal{X}_{n,r,v}^m(\delta,\rho)$ is a smooth manifold with a $C^{1,\alpha}$-Riemannian metric.
In section 5, we will construct $C^{1,\alpha}$-harmonic coordinate charts on $X$, and complete the proof of Theorem \ref{thm-bounded-local-covering-geometry}.
Theorem \ref{thm-Reifenberg-local-covering-geometry} and Theorem \ref{thm-fibration-of-C1-manifolds} will be proved in section 6 and 7 respectively.

{\bf Acknowledgement.} The authors are grateful to Professor Xiaochun Rong for his interest and very helpful discussion on the results in this paper. Z. J. is supported by China Postdoctoral Science Foundation Grant No. 8206300494.
S. X. is supported in part by Beijing Natural Science Foundation Grant No. Z19003 and National Natural Science Foundation of China Grant No. 11871349.

\section{Preliminaries}\label{preliminaries}
In this section, we will supply some notations and basic results that will be used through the rest of the paper.
\subsection{$C^{1,\alpha}$-Convergence}
In this subsection, we will introduce the concepts such as the harmonic radius of a Riemannian manifold, $C^{1,\alpha}$-convergence of a sequence of Riemannian manifolds. After that, we will give the well-known Cheeger-Gromov's and Anderson's $C^{1,\alpha}$-convergence theorems \cite{Cheeger1970}, \cite{GLP1981}, \cite{Anderson1990}, \cite{AnCh1992} (cf. also \cite{Peters1987}, \cite{GW1988}, \cite{Kasue1989}, \cite{Gao1990}, \cite{HebeyHerzlich1997}).
\begin{definition}\label{harmonic-radius}
Given $\alpha\in (0,1)$ and $Q>0$. Let $(M,g)$ be a smooth $n$-manifold with a $C^{1,\alpha}$-Riemannian metric $g$. For any $q\in M$, we define the $C^{1,\alpha}$-harmonic radius at $q$ as the largest number $r_{h}=r_{h}(\alpha,Q)(q,g)$ such that on the geodesic ball $B_{r_{h}}(q,g)$ of radius $r_{h}$ centered at $q$, there is a harmonic coordinate chart $\varphi=(x^{1},\dots,x^{n}):B_{r_{h}}(q,g)\to \Omega\subset\mathbb{R}^{n}$ such that the metric tensor admits the following $(\alpha,Q)$-$C^{1,\alpha}$-control:
\begin{enumerate}
	\item\label{def-hr-1} $e^{-Q}\delta_{ij}\leq g_{ij}\leq e^{Q}\delta_{ij}$ as bilinear forms, where $g_{ij}=g\left(\frac{\partial}{\partial x^{i}},\frac{\partial}{\partial x^{j}}\right)$, for $i,j=1,\dots,n$, and $\delta_{ij}$ are the Kronecker symbols,
	\item\label{def-hr-2} $r_{h}^{1+\alpha}\left\|\partial_{k}g_{ij}\right\|_{C^{0,\alpha}(\Omega)}\leq e^{Q}$, which means
	$$\sum_{k=1}^{n}r_{h}\sup_{x\in \Omega}\left|\partial_{k}g_{ij}(x)\right|+\sum_{k=1}^{n}r_{h}^{1+\alpha}\sup_{y,z\in \Omega, y\neq z}\frac{\left|\partial_{k}g_{ij}(y)-\partial_{k}g_{ij}(z)\right|}{d_{g}(y,z)^{\alpha}}\leq e^{Q},$$
	holds for $\partial_{k}=\frac{\partial}{\partial x^{k}}$ and the distance $d_{g}$ associated with $g$.
\end{enumerate}

The harmonic radius $r_{h}(\alpha,Q)(M,g)$ of $(M,g)$ is now defined by $r_{h}(\alpha,Q)(M,g)=\inf_{q\in M}r_{h}(q,g)$. For simplicity we will omit $\alpha$ and $Q$ when there is no confusion.
\end{definition}

In general, the $C^{1,\alpha}$-norms of the components, $g_{ij}$, of metric $g$ in the coordinates $\{x_i\}_{i=1}^{n}$ are defined on the Euclidean domain $\Omega$. For convenience, we also denote the $C^{1,\alpha}$-norm of $g_{ij}$ on $\Omega=\varphi(B_{r_{h}}(q,g))$ by $\left\|g_{ij}\right\|_{C^{1,\alpha}(B_{r_{h}}(q,g))}$.

Note that, $g\left(\frac{\partial}{\partial x^{i}},\frac{\partial}{\partial x^{j}}\right)$ and  $\partial_{k}g_{ij}$ in \eqref{def-hr-1} and \eqref{def-hr-2} can be replaced equivalently with $g(\nabla x^i,\nabla x^j)$ and $\nabla x^{k}\tilde{g}_{ij}$ respectively, where $\nabla x^{k}$ is the gradient of the coordinate function $x^k$ with respect to metric $g$.

\begin{remark}\label{scaling-invariant-harmonic-radius}
By definition, $r_{h}(q,g_\delta)=\delta^{-1}r_{h}(q,g)$ for $g_\delta=\delta^{-2}g$.
\end{remark}

Now we give the concept of $C^{1,\alpha}$-convergence of a sequence of Riemannian manifolds.
\begin{definition}\label{ck-close-1}
Let $M$ be a closed smooth $n$-manifold. Let $g_{i}$ and $g$ be complete $C^{1,\alpha}$-smooth Riemannian metrics on $M$. We say that $g_{i}$ converges to $g$ in the sense of $C^{1,\alpha}$-norm if for any $p\in M$, there exists a coordinate chart around $p$, $\left(x^{1},\dots,x^{n}\right):U\to \Omega\subset\mathbb{R}^{n}$, such that $g_{i,st}=g_{i}\left(\frac{\partial}{\partial x^{s}},\frac{\partial}{\partial x^{t}}\right)$ $C^{1,\alpha}$-converges to $g_{st}=g\left(\frac{\partial}{\partial x^{s}},\frac{\partial}{\partial x^{t}}\right)$ as $i\to \infty$.
i.e.,
\begin{equation}
\left\|g_{i,st}-g_{st}\right\|_{C^{1,\alpha}(\Omega)}\to 0, \ \text{as}\ i\to \infty,
\end{equation}
where the $C^{1,\alpha}$-norm $\left\|f\right\|_{C^{1,\alpha}(\Omega)}$ of a smooth function $f$ is defined by
$$\left\|f\right\|_{C^{1,\alpha}(\Omega)}=\sup_{x\in \Omega}\left|f(x)\right|+\sum_{k=1}^{n}\sup_{x\in \Omega}\left|\partial_{k}f(x)\right|
+\sum_{k=1}^{n}\sup_{y\neq z \in \Omega}\frac{|\partial_{k}f(y)-\partial_{k}f(z)|}{d_{g}(y,z)^{\alpha}}
$$
and $\partial_{k}=\frac{\partial}{\partial x^{k}}$.
\end{definition}
In practice, an open cover of coordinate charts $\left(x_{j}^{1},\dots,x_{j}^{n}\right):U_{j}\to \Omega_{j}\subset\mathbb{R}^{n}$ are usually fixed as the background coordinate charts for the $C^{1,\alpha}$-convergence on $M$.
\begin{definition}\label{C1-convergence-for-Rie-manifolds}
Let $(M_{j},g_{j})$ and $(M,g)$ be closed smooth  $n$-manifolds with $C^{1,\alpha}$-smooth Riemannian metrics. We say that $(M_{j},g_{j})$ converges to $(M,g)$ in the $C^{1,\alpha}$-topology if there exists an integer $j_{0}>0$ such that the following holds: for each $j\geq j_{0}$ there exists $C^{2,\alpha}$-diffeomorphism $\Phi_{j}:M\to M_{j}$ such that the pullback metric $\Phi_{j}^{*}g_{j}$ converges to $g$ in the sense of $C^{1,\alpha}$-norm.
\end{definition}
Note that, the pullback of $g_j$ by a diffeomorphism is crucial in Definition \ref{C1-convergence-for-Rie-manifolds}: even if $(M,g_j)$ converges to $(M,g)$ on the same manifold $M$ in the sense of $C^{1,\alpha}$-topology, it does not mean that $g_j$ converges to $g$ in the $C^{1,\alpha}$-norm. A counterexample can be found in \cite[Remark 3 below Main theorem]{HebeyHerzlich1997}.
\par

We say that a sequence of pointed complete Riemannian manifold $(M_i,g_i,p_i)$ converges in the $C^{1,\alpha}$-topology to a limit $(M,g,p)$ if
\begin{enumerate}
	\item there exists an exhaustion of $M$ by open subsets $\{U_i\}_{i=1}^{\infty}$ such that $U_i\subseteq U_{i+1}$ and $M=\bigcup U_i$;
	\item there exists a sequence of $C^{2,\alpha}$-embeddings $\phi_{i}:U_i\to M_i$ such that
	\begin{equation*}
	\phi_i(p)=p_i, \ \text{and}\ \phi_i^{*}g_i\overset{C^{1,\alpha}}{\longrightarrow} g
	\end{equation*}
	uniformly on any compact subset of $M$.
\end{enumerate}

Let us view $(\Omega,g_i)$ as the domain in Definition \ref{harmonic-radius} with the pullback metric by $\varphi_i^{-1}$, where $\varphi_i$ is a $C^{1,\alpha}$-harmonic coordinate chart $\varphi_i:B_{r_{h}}(q_i,g_i)\to \Omega\subset \mathbb R^n$. Then the Cartesian coordinates on $\mathbb R^n$ $(x^1,\dots,x^n):(\Omega,g_i)\to \mathbb R^n$ is harmonic with respect to $g_i$, which satisfies the $(\alpha,Q)$-$C^{1,\alpha}$-control \eqref{def-hr-1}-\eqref{def-hr-2}.
In the harmonic coordinates for a metric tensor $g$, the Ricci curvature satisfies the following equation:
$$g^{ij}\frac{\partial^2 g_{rs}}{\partial x^i\partial x^j}+B(\frac{\partial g_{kl}}{\partial x^m},g_{kl})=-2(\op{Ric}_{g})_{rs},$$
where $B$ is a quadratic term in $\frac{\partial g_{kl}}{\partial x^m}$ for $m=1,\dots,n$ (cf. \cite{DeturckKazdan1981}). By the standard $L^p$-estimate for elliptic PDEs, the $L^{2,p}$-norm of $g_{kl}$ admits a uniform bound that depends on $L^{1,p}$-norm of $g^{ij}$, $L^{p}$-bound on the term $B$ and the $L^{p}$-bound on $(\op{Ric}_g)_{rs}$ for any $1<p<+\infty$. Hence a subsequence of metric tensors $g_i$ converges to a limit $C^{1,\alpha}$-metric $g$ in the $C^{1,\alpha}$-norm if $|\op{Ric}_{g_i}|\leq n-1,$ where $\alpha=1-\frac{n}{p}$ for any $p>n$. More generally, if $g_i$ converges to $g$ in the $C^{1,\alpha}$-norm with respect to another fixed coordinates on $\Omega$, then similarly by the elliptic $L^p$ regularity, the harmonic coordinates $(x^1_i,\cdots,x^n_i)$ of $g_i$ admit a uniform $L^{3,p}$-bound, which implies that they converge to the harmonic coordinates of $g$ in the $C^{2,\alpha}$-norm.

Conversely, given a harmonic coordinate chart $(x^1,\dots,x^{n}):(\Omega,g)\to \mathbb{R}^{n}$ at $p$ with $(\alpha,Q)$-$C^{1,\alpha}$-control for $g$, the Dirichlet problem associated with $C^{1,\alpha}$-nearby metric $g_i$ can be solved on $(\Omega,g_i)$, with boundary value $x_{i}^{k}=x^{k}$ for each $k=1,\dots,n$. And the Schauder estimates give almost the same $C^{1,\alpha}$-control in the interior of $\Omega$.

Therefore, the $C^{1,\alpha}$-harmonic radius under bounded Ricci curvature is continuous in the sense of $C^{1,\alpha}$-topology, i.e., the following proposition.
\begin{proposition}[\cite{Anderson1990},\cite{AnCh1992}]\label{continuity-of-harmonic-radius}
Let $(M_{i},g_{i})$ be a sequence of Riemannian manifolds with $|\op{Ric}_{(M_i,g_i)}|\le n-1$, which $C^{1,\alpha}$-converges to a $C^{1,\alpha}$-Riemannian manifold $(M,g)$. Then
\begin{equation*}
r_{h}(M,g)=\lim_{i\to\infty}r_{h}(M_{i},g_i),
\end{equation*}
The same holds for $r_h(z_i,g_i)$ and $r_h(z,g)$ as $z_i\in (M_i,g_i)$ converges to $z\in (M,g)$.
\end{proposition}

Cheeger-Gromov-Anderson's $C^{1,\alpha}$-convergence theorem says that the diffeomorphism types of the whole manifolds are also stable under the $C^{1,\alpha}$-topology, which provides a fundamental tool in this paper.
\begin{theorem}[\cite{Anderson1990}]\label{thm-An-convergence}
	The space $\mathcal{M}_{|Ric|}^n(i_0,D)$ of all closed Riemannian $n$-manifolds $(M,g)$ such that
	\begin{equation}\label{condition-C1alpha-convergence}
	\left|\Ric_{(M,g)}\right|\le n-1,\quad \injrad(M,g)\ge i_0>0,\quad  \diam(M,g)\le D
	\end{equation}
is precompact in the $C^{1,\alpha}$-topology for any $0<\alpha<1$. More precisely, any sequence of $n$-manifolds $\left\{(M_i,g_i)\right\}\subseteq \mathcal{M}_{|Ric|}^n(i_0,D)$ admits a subsequence $(M_{i_1},g_{i_1})$ that converges to a closed smooth manifold $(M,g)$ with a $C^{1,\alpha}$-Riemannian metric $g$ via $C^\infty$-smooth diffeomorphisms $f_{i_1}: M\to M_{i_1}$ in the $C^{1,\alpha}$-topology.
\par
In particular, there are only finitely many diffeomorphism types of $n$-manifolds satisfying \eqref{condition-C1alpha-convergence}.
\end{theorem}

\begin{remark}\label{notes-of-C1-convergence}
Theorem \ref{thm-An-convergence} also holds for bounded domains in Riemannian manifolds, and for pointed complete but noncompact manifolds, after restricting to those compact subsets definitely away from the incomplete boundary (see \cite[Main Lemma 2.2]{Anderson1990}).
\end{remark}

\subsection{Ricci Flows}
Let $(M,g)$ be a closed Riemannian manifold. The Ricci flow was introduced by Hamilton \cite{Hamilton1982} as the solution of the following degenerate parabolic PDE,
\begin{equation}\label{ricci-flow-equation}
\frac{\partial}{\partial t}g(t)=-2\Ric_{g(t)},\qquad g(0)=g.
\end{equation}
The solution always exists for a short time $t>0$, and if it admits a finite maximal flow time $T_{max}<+\infty$, then the curvature tensor blows up as $t\to T_{max}$, i.e., $\max\left|\operatorname{Rm}(g(t))\right|_{g(t)}\to +\infty$.
\par
A basic property of Ricci flow is that it improves the regularity of the initial metric (\cite{Shi1989-1, Shi1989-2}), which depends on the flow time. The existence of a uniform definite flow time is important in practice.

Dai-Wei-Ye \cite{DWY1996} proved that a uniform flow time $T(n,r_0)$ exists for a closed $n$-manifold $(M,g)$ satisfying $\left|\Ric_{(M,g)}\right|\leq n-1$ and the conjugate radius $\conjrad(M,g)\geq r_{0}>0$. As already pointed out by \cite{CRX2017}, the conjugate radius condition in their proof is only used to derive a uniform $L^{2,p}$-harmonic coordinates for the lifted metric on $B_{r_0}(0)\subseteq T_{x}M$ for all $p\geq 1$ (see \cite[Remark 1]{DWY1996}) and $x\in M$, which is required to apply the weak maximum principle \cite[Theorem 2.1]{DWY1996}. Since the same holds at a preimage point $\tilde x$ on the universal covering space of a $\rho$-ball $B_\rho(x)$ on $(M,g)$ when $(M,g)$ has $(\delta,\rho)$-local covering geometry, \cite[Theorem 1.1]{DWY1996} can be reformulated into the following form.
\begin{theorem}[\cite{DWY1996}, cf. {\cite[Theorem 1.5]{CRX2017}}]\label{thm-smoothing-ricci-flow}
	Given $n,\rho>0$, there exist constants $\delta(n),T(n,\rho)>0$ and $C(n,\rho)>0$ such that for any $0<\delta\le \delta(n)$, if $(M,g)$ is a closed $n$-manifold
	with $|\Ric_M|\leq n-1$ and $(\delta,\rho)$-Reifenberg local covering geometry, then the Ricci flow equation (\ref{ricci-flow-equation})
	has a unique smooth solution $g(t)$ for $0<t\leq T(n,\rho)$ satisfying
\begin{equation}\label{ineq-smoothing-sec}
	\left\{
	\begin{array}{llll}
	& \left|g(t)-g\right|_{g}\leq 4t;\\
	& \left|\operatorname{Rm}(g(t))\right|_{g(t)}\leq C(n,\rho)t^{-\frac{1}{2}};\\
    & \left|\nabla^{k}\operatorname{Rm}(g(t))\right|_{g(t)}\leq C(n,\rho,k,t);\\
	& \left|\Ric(M,g(t))\right|_{g(t)}\leq 2(n-1),
	\end{array}
	\right.
	\end{equation}
where $\operatorname{Rm}(g(t))$ denotes the curvature tensor of $g(t)$, $\nabla^{k}\operatorname{Rm}(g(t))$ the $k^{th}$-convariant derivative of $\operatorname{Rm}(g(t))$, whose norm is measured in $g(t)$.
\end{theorem}

\subsection{Gromov-Hausdorff precompactness for the covering spaces of open balls}

Let $(M_i,g_i)$ be a sequence of complete Riemannian $n$-manifolds with $\op{Ric}_{(M_i,g_i)}\ge -(n-1)$. Let $B_r(x_i,g_i)$ be an open $r$-ball in $(M_i,g_i)$, and $\widetilde{B}(x_i,r)$ the Riemannian universal cover of $B_r(x_i,g_i)$. Then $\widetilde{B}(x_i,r)$ may not admit a convergence subsequence in the pointed Gromov-Hausdorff topology (see \cite[Example 3.2]{SormaniWei2004}).

Recently the third author \cite{Xu2021} proved a precompactness principle for open domains in complete Riemannian manifold with $\op{Ric}\ge -(n-1)$, which particularly is suitable for the covering spaces of open balls.

Let $\widehat{B}(x,r,R)$ be a component of the preimage $\pi^{-1}(B_r(x,g))$ in the Riemannian universal cover $\pi:\widetilde{B}(x,R)\to B_R(x,g)$. Then
$\widehat{B}(x,r,R)$ is a normal $G$-cover of $B_r(x,g)$, where
\begin{align*}
&G=\Gamma_{r/R,R}(x)=\op{Image}[\pi_1(B_r(x,g),x)\to \pi_1(B_R(x,g),x)], \quad\text{and}\\
&\pi_1(\widehat{B}(x,r,R))=\op{Kernel}[\pi_1(B_r(x,g),x)\to \pi_1(B_R(x,g),x)].
\end{align*}
We endow $\widehat{B}(x,r,R)$ with a base point $\hat x$ in the preimage of $x$ and its length metric induced from $\widetilde{B}(x,R)$.

\begin{theorem}[\cite{Xu2021}]\label{thm-precompactness-a}
	For any $R>r>0$, let $\widehat{\mathcal B}(r,R)$ be the set consisting of the Riemannian normal covers $\widehat{B}(x,r,R)$ of all open balls in complete Riemannian $n$-manifolds with $\op{Ric}\ge -(n-1)$.
	Then $\widehat{\mathcal B}(r,R)$ is precompact in the pointed Gromov-Hausdorff topology.
\end{theorem}

More generally, let $W(r)=\cap_jB_{r}(p_j,g)$ be a non-empty intersection of open $r$-balls $B_{r}(p_j,g)$ in $(M,g)$, and $\widetilde{W}(r)$ be the Riemannian universal cover of $W(r)$.
For $R>r>0$, we define $\widehat{W}(r,R)$ to be a component of the preimage of $W(r)$ in the Riemannian universal cover $\pi:\widetilde{W}(R)\to W(R)=\cap_jB_R(p_j,g)$.

\begin{theorem}[\cite{Xu2021}]\label{thm-precompactness-b}
	For any $R>r>0$, let $\widehat{\mathcal W}(r,R)$ be the set consisting of the normal covers $\widehat{W}(r,R)$ endowed with length metric of $W(r)$ in complete Riemannian $n$-manifolds with $\op{Ric}\ge -(n-1)$.
	Then $\widehat{\mathcal W}(r,R)$ is precompact in the pointed Gromov-Hausdorff topology.
\end{theorem}

Note that for each $j$, a component of $\pi_{j}^{-1}(W(R))$ in the universal cover $\pi_{j}:\widetilde{B}(p_{j},R)\to B_R(p_{j},g)$ is a normal cover of $W(R)$, such that $\widetilde{W}(R)$ covers $\pi_{j}^{-1}(W(R))$. By the definition of $\widehat{W}(r,R)$, we derive that
\begin{equation}\label{intersection-normal-cover}
\text{$\widehat{W}(r,R)$ is a normal cover of $\pi_{j}^{-1}(W(r))\subset \widehat{B}(p_{j},r,R)$.}
\end{equation}
This fact will be applied in Section 4.
\begin{remark}\label{rem-Reifenberg-normal-cover}
	Since $\widehat{B}(x,r/2,r)$ is a component of the preimage $\pi^{-1}(B_{r/2}(x,g))$ in the Riemannian universal cover $\pi:\widetilde{B}(x,r)\to B_r(x,g)$, the $(\delta,r)$-Reifenberg local covering geometry condition is naturally passed to $\widehat{B}(x,r/2,r)$. That is, if for any $x\in M$, $\tilde x\in \widetilde{B}(x,r)$ is a $(\delta,r)$-Reifenberg point, then $\hat x\in \widehat{B}(x,r/2,r)$ is $(\delta,r/4)$-Reifenberg.
\end{remark}

Finally we recall that, by \cite{HKRX2020} the local covering geometry via rewinding volume is equivalent to that by the Reifenberg condition for a manifold that is locally close to be Euclidean.
\begin{lemma}[{\cite[Lemma 2.1]{HKRX2020}}]\label{lem-rewinding-reifenberg}
	Given positive integer $n$ and real numbers $\delta, v>0$, there are constants $\epsilon(n), \rho(n,v,\delta)>0$ such that the following holds.
	
	Let $(M,g)$ be a Riemannian $n$-manifold with $\op{Ric}_{M}\ge -(n-1)$ and $$\op{Vol}(B_1(\tilde x))\ge v>0,$$ where $\tilde x$ is a preimage point of $x$ in the universal cover of $B_1(x)$.
	If $$d_{GH}(B_1(x),B_1^m(0))\le \epsilon(n),$$ then $\tilde x$ is a $(\delta,\rho(n,v,\delta))$-Reifenberg point, i.e.,
	$$d_{GH}(B_r(\tilde x), B_r^n(0))\le \delta r,\qquad \forall\; 0<r\le \rho(n,v,\delta).$$
\end{lemma}
\begin{proof}
	Lemma \ref{lem-rewinding-reifenberg} essentially is a restatement of \cite[Lemma 2.1]{HKRX2020}. For the reader's convenience we give a simple proof from a different viewpoint.
	
	Let us argue by contradiction. Assume that there is a sequence $(M_i,g_i)$ with $\op{Ric}_{M_i}\ge -(n-1)$, $\op{Vol}(B_1(\tilde x_i))\ge v>0$, $d_{GH}(B_1(x),B_1^m(0))\le \epsilon_i\to 0$, but $\tilde x_i$ is not a $(\delta,\rho)$-Reifenberg point for any fixed $\delta,\rho>0$ for all sufficient large $i$.
	
	By passing to a subseqence, let us consider the equivariant pointed Gromov-Hausdorff convergence as in \eqref{eqGH-local-covers}
	$$\begin{CD}
	(\widehat{B}(x_i,1/2,1),\hat x_i, \Gamma_i) @>GH>> (Y, \hat x, G)\\
	@V\pi_iVV @VV\pi V\\
	(B_{1/2}(x_i,g_i),x_i) @>GH>> (B_{1/2}^m(0),0)\subset \mathbb R^m.
	\end{CD}$$
	
	Since $Y/G=B_{1/2}^m(0)$, a standard blowing-up argument implies that the quotient of any tangent cone $T_{\hat x}$ modular the infinitesimal actions $dG$ induced by $G$ is $\mathbb R^m$, where the lines on $\mathbb R^m$ can be lifted onto $T_{\hat x}$. At the same time, by the fact that $Y$ is a Ricci-limit space of a non-collapsing sequence, $T_{\hat x}$ is a metric cone. Hence $T_{\hat x}$ splits to $\mathbb R^m\times C(\Sigma)$, and $dG$ acts transitively on $C(\Sigma)$. Hence $T_{\hat x}$ is $\mathbb R^n$.
	
	In particular, there is $\rho=\rho(\delta,\hat x)>0$ such that $\hat x$ is a $(\delta/10,\rho)$-Reifenberg point. Then by Colding's volume convergence (\cite{Colding1997}), $\tilde x_i$ is $(\delta,\rho)$-Reifenberg point for $i$ large, a contradiction.	
\end{proof}

\section{$C^{1,\alpha}$-regularity of limit spaces under bounded sectional curvature}
As the starting point of this paper, we give a simple proof of Theorem \ref{thm-Reifenberg-local-covering-geometry} for the limit spaces of closed manifolds with bounded sectional curvature.

In fact, it is essentially a corollary of \cite[Theorem 0.9, Corollary 0.11]{Fukaya1988}, which now can be improved as follows.
\begin{theorem}\label{thm-limit-bounded-sec}
	Let $(M_i,g_i)$ be a sequence of closed Riemannian $n$-manifolds with $|\op{sec}_{g_i}|\le 1$ such that $(M_i,g_i)\overset{GH}{\longrightarrow}X$. Then $X$ admits a stratification $X=S_0(X)\supset S_1(X)\supset \cdots \supset S_k(X)$ for $0\le k\le n$ (some of them may be the same) such that
	\numberwithin{enumi}{theorem}
	\begin{enumerate}[leftmargin=0pt,itemindent=*]
		\item\label{limit-bounded-sec-1} the Hausdorff dimension of $X$ is equal to $k$.
		\item\label{limit-bounded-sec-2} $S_{i}(X)\setminus S_{i+1}(X)=(R^{k-i}(X),h_{k-i})$ is a smooth $(k-i)$-manifold with a $C^{1,\alpha}$-Riemannian metric.
		
		Moreover, if the $k$-Hausdorff measure $H^k(X)\ge v>0$ and the diameter of $X$ $\le D$, then for any point $p\in S_{i}(X)\setminus S_{i+1}(X)$ that is $\epsilon$-away from $S_{i+1}(X)$, the $C^{1,\alpha}$-harmonic radius at $p$ is no less than $r_h(\epsilon\,|\,n,r,\alpha,Q,v,D)>0$, where $Q>0$, $0<\alpha<1$.
	\end{enumerate}

\end{theorem}

Note that the original proof of \cite[Theorem 0.9]{Fukaya1988} depends on orthonormal frame bundles of $(M_i,g_i)$, which do not generally admit a uniform sectional curvature bound such that only a $C^{0,\alpha}$-regularity can be derived on $X$. In the following we will give a different and simple proof from the view point of local normal covers.

The key point behind is the following observation \cite[Lemma 7.2]{Fukaya1988} by Fukaya.

\begin{lemma}[{cf. \cite[Lemma 7.2]{Fukaya1988}}]\label{lem-curvature-group-action-*}
	Let $(M,g)$ be a Riemannian manifold whose sectional curvature $b\le \op{sec}_g\le a$. Assume that there is a proper and free isometric action by $G$ on $(M,g)$. Let
	\begin{equation}\label{def-rate-normal-deviatiion-*}
	(r/t)^\perp_p(G)=\sup \left\{\left.\|r^\perp_p(g)\|/d(p,g(p))\,\right|\, g\neq e\in G, \; \|r_p^\perp(g)\|\text{ is well-defined} \right\},
	\end{equation}
	where $\|r^\perp_p(g)\|=\sup \{\measuredangle(dg(v),P(v))\,|\, v\in T_p^\perp G(p)\}$ is defined when $d(p,g(p))$ is less than the injectivity radius at $p$, and $P$ is the parallel transport from $T_{p}M$ to $T_{g(p)}M$ along the unique minimal geodesic.
	
	Then the sectional curvature of the quotient $M/G$ at $\pi(p)$ is bounded by
	$$b \le \op{sec}_{\pi(p)}\le a + 6((r/t)^\perp_p(G))^2.$$
\end{lemma}
\begin{proof}
	Here the only difference from original \cite[Lemma 7.2]{Fukaya1988} is the rate of infinitesimal angle deviation \eqref{def-rate-normal-deviatiion-*}. In \cite[Lemma 7.2]{Fukaya1988} the same conclusion was proved for $$(r/t)_p(G)=\sup \left\{\left.\frac{\|r_p(g)\|}{d(p,g(p))}\,\right|\, g\neq e\in G, \|r_p(g)\|\text{ is well-defined} \right\},$$
	where $\|r_p(g)\|=\sup \{\measuredangle(dg(v),P(v))\,|\, v\in T_pM\}$ is defined when $d(p,g(p))$ is less than the injectivity radius at $p$.
	
	Note that by the tubular neighborhood theorem (e.g., see \cite{Grove2000}), only the part of $dg$ on the normal space $T_p^\perp G(p)$ to the orbit $G(p)$ makes effect on the horizontal directions. Lemma  \ref{lem-curvature-group-action-*} follows the proof of \cite[Lemma 7.2]{Fukaya1988} line-by-line.
\end{proof}

\begin{proof}[Proof of Theorem \ref{thm-limit-bounded-sec}]
	~
	
	Let us consider the smoothed metrics $g_i(t)$ on $(M_i,g_i)$, which are the solutions of Hamilton's Ricci flow equation with the initial condition $g_i(0)=g_i$. By \cite{Shi1989-1} (cf. \cite{Rong1996}), $g_i(t)$ also admits a uniform sectional curvature bound $1+C(n)t$ and a uniform higher regularities \eqref{ineq-smoothing-sec}.
	
	By passing to a subsequence, we assume that $(M_{i},g_{i}(t))\overset{GH}{\longrightarrow}X_t$ as $i\to \infty$ for any fixed $t>0$.
	Then $X_t$ is $e^{2t}$-bi-Lipschitz to $X$, and  admits a stratification
	$S_0(X_t)\supset S_1(X_t)\supset \cdots \supset S_k(X_t)$ for $0\le k\le n$ such that each strata $R^{k-i}=S_i(X_t)\setminus S_{i+1}(X_t)$ is a smooth Riemannian $(k-i)$-manifold with sectional curvature $\ge 1+\varkappa(t\,|\,n)$.
		
	Indeed, for $x_i\in (M_i,g_i(t))$ that approaches $x\in X_t$, let us consider the equivariant limit spaces $(Y_t,  g^*(t), G_t)$ of the $\frac{\pi}{2}$-ball in the tangent space $T_{x_i}M_i$ of $x_i\in (M_i,g_i(t))$ with the pullback metric $g_i^*(t)$ via the exponential map:
	\begin{equation}\label{graph-tangent-space}
	\begin{CD}
	(B_{\frac{\pi}{2}}(0_i), g_i^*(t), 0_i, \Gamma_i) @>GH>i\to \infty> (Y_t, g^*(t), y_t, G_t)\\
	@V\exp VV @V\pi_t VV\\
	(B_{\frac{\pi}{2}}(x_i,g_i(t)), x_i) @>GH>i\to \infty> B_{\frac{\pi}{2}}(x,X_t)=Y_t/G_t,
	\end{CD}
	\end{equation}
	where $\Gamma_i$ is the pseudo-group action by the local fundamental group of $B_{\frac{\pi}{2}}(x_i,g_i(t))$, and $B_{\frac{\pi}{2}}(x,X_t)$ is equipped with its length metric. Then $Y_t$ is a smooth Riemannian manifold with sectional curvature $|\op{sec}_{Y_t}|\le 1+C(n)t$.
	
	Since the pseudo-group $G_t$ acts on $Y_t$ by isometries, by the standard theory of isometric actions on Riemannian manifolds (e.g., \cite{Alexandrino2015}) the orbit space $Y_t/G_t$ admits a standard stratification by isotropy types. So is $X_t$. Hence we derive \eqref{limit-bounded-sec-1}.
	
	For \eqref{limit-bounded-sec-2}, it suffices to show that the sectional curvature at points in $S_i(X_t)$ $\epsilon$-definitely away from $S_{i+1}(X_t)$ is bounded uniformly by a constant $C(n,v,D)$. Then \eqref{limit-bounded-sec-2} for the original limit space $X$ immediate follows from Cheeger-Gromov's convergence theorem applied on $X_t$ as $t\to 0$.
	
	Let us argue by contradiction. Suppose there is a sequence of limit spaces $(X_{j,t},x_j)$ such that $x_j\in R^{k-i}(X_{j,t})$, $d(x_j, S_{i+1}(X_{j,t}))\ge \epsilon$, and $t=t_j\in(0,T(n)]$ is arbitrary chosen, but the sectional curvature of $R^{k-i}(X_{j,t})$ at $x_j$ is unbounded as $j\to \infty$. Since $X_t$ is an Alexandrov space with curvature $\ge C(n)$,
	$$\max_{v,w\in T_{x_j}} \op{sec}_{X_{j,t}}(v\wedge w)\to \infty \text{ as $j\to +\infty$}.$$
	
	By passing to a subsequence, let us consider the equivariant convergence of limit spaces $(Y_{j,t},y_j)$ in \eqref{graph-tangent-space} for each $X_{j,t}$.
	\begin{equation}\label{graph-tangent-limit}
	\begin{CD}
	(Y_{j,t}, g_{j}^*(t), y_j, G_{j,t}) @>C^{1,\alpha}>j\to \infty> (Y, g^*, y, G)\\
	@V\pi_{j,t}VV @V\pi_{\infty} VV\\
	(B_{\frac{\pi}{2}}(x_j,X_{j,t}), x_j) @>GH>j\to \infty> (Y/G, x).
	\end{CD}
	\end{equation}
	
	In order to apply Lemma \ref{lem-curvature-group-action-*}, we first point out that the Lie group $G_{j,t}$ can be reduced to the case of free action.
	
	Indeed, let $G_{j,0}$ be the identity component of  $G_{j,t}$. Instead of $G_{j,t}$, we consider the actions by $G_{j,0}$, which by the Heintze-Margulis lemma (e.g., \cite[Lemma 4.1]{Fukaya1988}, cf. \cite{Gromov1978}, \cite{BuserKarcher81}) is a niloptent Lie group. Moreover, by \cite[Lemma 5.1]{Fukaya1988} the isotropy group $G_{j,0,y_j}$ of $G_{j,0}$ lies in the center of $G_{j,0}$. Hence the isotropy group of $G_{j,0}$ in a conjugacy class is unique on $Y_{j,t}$. It follows that the union of all orbits of the same isotropy type $G_{j,0,y_j}$, $\pi_{j,t}^{-1}(R^{k-i}(Y_{j,t}/G_{j,t}))$, is the fixed-point set of $G_{j,0,y_j}$, and hence is totally geodesic. Moreover, the isometric action $G_{j,t}$ on $\pi_{j,t}^{-1}(R^{k-i}(Y_{j,t}/G_{j,t}))$ can be reduced to the quotient group $G_{j,0}/G_{j,0,y_j}$ that acts effectively and freely.
	
	By our assumption, $Y_{j,t}/G_{j,0}$ is not collapsing. For simplicity we assume that the action of $G_{j,0}$ itself is free.
	
	We claim that
	\begin{enumerate}
		\item \label{claim-group-deviation-1} there is a sequence $\gamma_j\in G_{j,0}$ such that $\frac{\|r^\perp_{y_j}(\gamma_j)\|}{d(y_j,\gamma_j(y_j))}\to\infty $. By a suitable choice of $n_j\to \infty$, $d(\gamma_j^{n_j}(y_j),y_j)\to 0$ and $\left.d\gamma_j^{n_j}\right|_{y_j}$ maps the normal space of $G_{j,0}(y_j)$ by a uniform and definite deviation away from the parallel transformation.
		
		\item \label{claim-group-deviation-2} the normal space of $G_{0}(y)$ has the same dimension as that of $G_{j,0}(y_j)$.
		
		\item \label{claim-group-deviation-3} the limit $\eta$ of $\gamma_j^{n_j}$ lies in the isotropy group $G_{0,y}$ at $y$, whose differential admits a non-trivial transformation on the normal space of $G_{0}(y)$.
	\end{enumerate}
	
	The claims above will yield a contradiction immediately. Indeed, let $\eta_s=\lim \gamma_{j}^{[sn_j]}$ for each $s\in [0,1]$ and $\gamma_j$ in \eqref{claim-group-deviation-1}. Then $\eta_s$ forms a continuous path in $G_{0,y}$ which by \eqref{claim-group-deviation-3} acts on $T_{y}^\perp G_{0}(y)$ non-trivially. By \eqref{claim-group-deviation-2} and the slice theorem, $Y/G_0$ has lower dimension than $Y_{j,t}/G_{j,0}$, a contradiction to that they are assumed to be non-collapsing.
	
	The verification of \eqref{claim-group-deviation-1}-\eqref{claim-group-deviation-3}:
	
	By Lemma \ref{lem-curvature-group-action-*}, the unboundedness on the upper curvature at $x_j\in X_{j,t}$ implies that $(r/t)_{y_j}^\perp(G_{j,0})$ blows up as $j\to \infty$. Hence \eqref{claim-group-deviation-1} holds by the definition \eqref{def-rate-normal-deviatiion-*} of $(r/t)_{p}^\perp(G)$.
	
	Since $(Y_{j,t},g_j^*(t))$ converges to $(Y,g^*)$ in the $C^{1,\alpha}$-topology, let us identify $B_{\frac{\pi}{4}}(y_j,g_j^*(t))$ with $B_{\frac{\pi}{4}}(y,g^*)$ via a suitable diffeomorphism.  Then the isometric actions by one-parameter subgroups of $G_{j,0}$ $C^1$-converges to that of $G_0$ on $B_{\frac{\pi}{4}}(y,g^*(t))$, which implies $\dim G_0(y)\ge \dim G_{j,0}(y_j)$.
	
	Because $Y_{j,t}/G_{j,0}$ is assumped to be non-collapsing, the orbit $G_{j,0}(y_j)$ and $G_{0}(y)$ must have same dimension, which implies \eqref{claim-group-deviation-2}. Moreover, the normal space of $G_{j,0}(y_j)$ converges to that of $G_{0}(y)$, and the limit of $\gamma_j^{n_j}$ in \eqref{claim-group-deviation-1} is a nontrivial element $\eta$ in $G_{0,y}$ with a non-trivial deviation on the normal space of $G_{0}(y)$, i.e., \eqref{claim-group-deviation-3}.
\end{proof}

We point out that, though the manifolds $(M_i,g_i)$ with $|\Ric_{(M_i,g_i)}|\leq n-1$ and $(\delta,r)$-Reifenberg local covering geometry still can be smoothed to $g_i(t)$ via Ricci flow by Theorem \ref{thm-smoothing-ricci-flow}, the proof of Theorem \ref{thm-limit-bounded-sec} fails to work for their limit spaces, due to that the sectional curvature of $g_i(t)$ generally blows up as $t\to 0$. Therefore, instead of Cheeger-Gromov's convergence theorem, we have to construct a harmonic coordinates directly in the next two sections.
\section{Construction of the $C^{1,\alpha}$-harmonic coordinate chart in quotient spaces}
In this section we prove the technical theorem \ref{tech-theorem-*}.

We first prove an adapted harmonic coordinate chart descends to a chart on the base manifold that is almost harmonic.

 \begin{lemma}\label{almost-C1-harmonic-coordinate-chart-on-X}
 Let the assumptions be as in Theorem \ref{tech-theorem-*}. Assume that \begin{equation}\label{Hessian-control-for-yj-**}
 \|\operatorname{Hess}y^j\|_{C^{0,\alpha}(B_{r}(p))}\leq \tau, \qquad j=1,\dots,m.
 \end{equation}
Then the adapted coordinate chart $(y^1,\dots, y^n): B_r(p)\to \mathbb R^n $ at $p$ descends to a coordinate chart $(z^1,\dots, z^m):B_{2r/3}(\bar{p})\to \mathbb{R}^{m}$ at $\bar p=\pi(p)$ such that $y^j=z^j\circ \pi$, and the metric tensor $h$ on $B_{2r/3}(\bar{p})\subset X$ expressed in $h_{st}=h(\bar\nabla z^s,\bar\nabla z^{t})$ satisfy
 \numberwithin{enumi}{theorem}
 \begin{enumerate}
 \item\label{C0-estimate-for-h}
 $e^{-Q}\delta_{st}\leq h_{st}\leq e^{Q}\delta_{st}$,
 \item\label{C1-estimate-for-h}
 $(2r/3)^{1+\alpha}\left\|\bar\nabla z^{j} h_{st}\right\|_{C^{0,\alpha}\left(B_{2r/3}(\bar{p})\right)}\leq e^{Q}$, and
 \item\label{control-of-Laplace-of-zj} $z^j$ ($j=1,\dots,m$) is almost harmonic in the sense that
 $$\left\|\op{Hess}_h z^{j}\right\|_{C^{0,\alpha}\left(B_{2r/3}(\bar{p})\right)}\leq \Psi(\tau\,|\,n,\alpha,Q),$$
 \end{enumerate}
where $\bar \nabla z^j$ is the gradient of $z^j$ with respect to $h$, the $C^{0,\alpha}$-norm is taken in the coordinate chart $(z^1,\dots, z^m)$, and $\Psi(\tau\,|\,n,\alpha,Q)$ is a function depending on $\tau, n, \alpha, Q$ such that $\Psi(\tau\,|\,n,\alpha,Q)\to 0$ as $\tau\to 0$ with fixed $n,\alpha,Q$.
 \end{lemma}
 \begin{proof}

 \par
 Because $y^j$ ($j=1,\dots,m$) takes the same value on each fiber, $z^j=y^j\circ\pi^{-1}$ is well-defined.

 Let $g_{st}=g(\nabla y^{s},\nabla y^{t})$, where $\nabla$ is the Levi-Civita connection on $(Y,g)$. From the definition of $C^{1,\alpha}$-harmonic coordinate chart, one has
 \begin{equation}\label{C1-norm-for-metric-g-1}
 \begin{cases}
 e^{-Q}\delta_{st}\leq g_{st}\leq e^{Q}\delta_{st};\\
 r^{1+\alpha}\|\partial_{j}g_{st}\|_{C^{0,\alpha}(B_{r}(p))}\leq e^{Q},
 \end{cases}
 \end{equation}
where $\partial_{j}=\nabla y^{j}$ and the $C^{0,\alpha}$-norm is taken in the coordinates $\{y^{j}\}$.

Since $\pi$ is a Riemannian submersion, the gradient $\nabla y^{j}$ is a horizontal vector field on $Y$ and
\begin{equation}\label{pi-relation-vector-field}
\pi_{*}(\nabla y^{j})=\bar{\nabla}z^{j}, \quad j=1,\dots,m,
\end{equation}
where $\bar{\nabla}$ denotes the Levi-Civita connection on $(X,h)$, and $\pi_*$ is the tangent map of $\pi$. Define $h_{st}=h(\bar{\nabla} z^{s},\bar{\nabla}z^{t})$. Then
\begin{equation}\label{pi-relation-metric-tensor}
h_{st}=g_{st}, \quad s,t=1,\dots, m.
\end{equation}
It follows from (\ref{C1-norm-for-metric-g-1}) and \eqref{pi-relation-metric-tensor} that $h_{st}$ satisfies (\ref{C0-estimate-for-h}), and thus  $(z^{1},\dots,z^{m}):B_{2r/3}(\bar{p})\to \mathbb{R}^{m}$ is a coordinate chart at $\bar{p}$.

What remains is to verify the estimates (\ref{C1-estimate-for-h}) and (\ref{control-of-Laplace-of-zj}). First, by (\ref{pi-relation-vector-field}) and \eqref{pi-relation-metric-tensor}, we derive that for each $j=1,\dots,m,$
\begin{equation}\label{C1-estimate-using-Riemmanian-submersion}
\begin{aligned}
\bar{\nabla} z^{j}(h_{st})(\bar{q})&
=\pi_{*}(\nabla y^{j})(h(\bar{\nabla} z^{s},\bar{\nabla}z^{t}))(\pi(q))
=\nabla y^{j}(g(\nabla y^{s},\nabla y^{t}))(q)\\&
=\text{Hess } y^{s}(\nabla y^{j},\nabla y^{t})(q)+
\text{Hess } y^{t}(\nabla y^{s},\nabla y^{j})(q)
\end{aligned}
\end{equation}
for any $\bar{q}\in B_{2r/3}(\bar{p})$ and $q\in \pi^{-1}(\bar{q})$. Then \eqref{Hessian-control-for-yj-**} together with (\ref{C1-estimate-using-Riemmanian-submersion}) yields (\ref{C1-estimate-for-h}).
\par
Secondly, by (\ref{pi-relation-vector-field}), we have
\begin{equation}\label{formula-for-Riemannian-submersion-1}
\nabla_{\nabla y^{k}}\nabla y^{l}=\widetilde{\bar{\nabla}_{\bar{\nabla}z^{k}}\bar{\nabla}z^{l}}
+A(\nabla y^{k},\nabla y^{l}),\qquad \text{for any $k,l=1,\dots,m$}
\end{equation}
where $A(\cdot,\cdot)=[\cdot, \cdot]^\top$ denotes the horizontal integral tensor of Riemannian submersion $\pi$ that takes values tangent to the fibers, and for any smooth vector field $Z$ on $X$, $\tilde{Z}$ denotes its horizontal lifting on $Y$.
Combing (\ref{pi-relation-vector-field}) with (\ref{formula-for-Riemannian-submersion-1}) and by the fact that $\pi$ is Riemannian submersion, we have
\begin{equation}\label{Hess-zj-equal-to-Hess-yj}
\operatorname{Hess}z^{j}(\bar{\nabla}z^{k},\bar{\nabla} z^{l})
=\operatorname{Hess}y^{j}(\nabla y^{k},\nabla y^{l}) \qquad \text{for any $j,k,l=1,\dots,m$.}
\end{equation}
Now (\ref{Hess-zj-equal-to-Hess-yj}) together with (\ref{Hessian-control-for-yj-**}) yields (\ref{control-of-Laplace-of-zj}). This complete the proof of Lemma \ref{almost-C1-harmonic-coordinate-chart-on-X}.
 \end{proof}

\begin{remark}\label{rem-regularity-passed-by-submersion}
In the proof of Lemma \ref{almost-C1-harmonic-coordinate-chart-on-X}, it follows from \eqref{pi-relation-metric-tensor}-\eqref{C1-estimate-using-Riemmanian-submersion} that the curvature tensor of $(X,h)$ and its covariant derivatives of any order satisfy the same regularity as $(Y,g)$ up to a definite ratio depending on $n$ and $Q$. We will apply the fact in proving (\ref{thm-bounded-local-covering-geometry}.1.a-c) in Theorem \ref{thm-bounded-local-covering-geometry}.
\end{remark}

Next, let us prove Theorem \ref{tech-theorem-*} by solving the Dirichlet problem with the boundary condition $z^j$.
\begin{proof}[Proof of Theorem \ref{tech-theorem-*}]~
\par Let $\varphi:=(z^{1},\dots,z^{m}):B_{2r/3}(\bar{p})\to \mathbb{R}^{m}$ be a coordinate chart at $\bar{p}$ provided by Lemma \ref{almost-C1-harmonic-coordinate-chart-on-X}.
Let $0\in \mathbb{R}^{m}$ be the origin. By a shift in value, we assume that $\varphi(\bar{p})=0$. Let $\Omega\subseteq \mathbb{R}^{m}$ denote the image set of $\varphi$ and $\varphi^{-1}:\Omega\to (B_{2r/3}(\bar{p}),h)$ the inverse map of $\varphi$. Let us pullback the metric $h$ on $X$ to $\Omega$ by $\varphi^{-1}$, where we still denote $(\varphi^{-1})^*h$ on $\Omega$ by $h$, and identify $(B_{2r/3}(\bar{p}),h)$ with $(\Omega,h)$ by $\varphi^{-1}$.

From (\ref{C0-estimate-for-h}) and by $0<Q\le 10^{-2}$, we have the Euclidean ball $B_{3r/5}^m(0)\subseteq \Omega$. Let $x^{j}$ be the solution of the following Dirichlet problem:
\begin{equation*}
\begin{cases}
 \Delta_{h} x^j=0, &\text{in}\ B_{3r/5}^{m}(0);\\
 x^j=z^j, &\text{on}\ \partial B_{3r/5}^{m}(0).
 \end{cases}
  \end{equation*}
Note that (\ref{C0-estimate-for-h}) yields that $B_{r/2}(\bar{p})\subseteq \varphi^{-1}(B_{3r/5}^{m}(0))$.
We claim that
\begin{equation}\label{desired-harmonic-coordinate-chart}
(x^1,\dots,x^m):\varphi^{-1}(B_{3r/5}^{m}(0))\to \mathbb{R}^{m}
\end{equation}
yields the desired harmonic coordinate chart in Theorem \ref{tech-theorem-*}, i.e.,
the following $C^{1,\alpha}$-estimates hold:
\begin{equation}\label{C1-norm-for-metric-h}
\begin{cases}
e^{-2Q}\delta_{st}\leq \bar{h}_{st}\leq e^{2Q}\delta_{st};\\
(r/2)^{1+\alpha}\left\|\frac{\partial}{\partial x^{j}}\bar{h}_{st}\right\|_{C^{0,\alpha}(B_{r/2}(\bar{p}))}\leq e^{2Q},
\end{cases}
\end{equation}
where $\bar{h}_{st}=h(\frac{\partial}{\partial x^{s}},\frac{\partial}{\partial x^{t}})$ and $C^{0,\alpha}$-norm is taken in the coordinates $\{x^j\}$.

Indeed, let us consider the functions $w^j=x^j-z^j$, which satisfy the following equation:
\begin{equation*}
\begin{cases}
 \Delta_{h} w^j= \Delta_{h} z^j, &\text{in}\ B_{3r/5}^{m}(0);\\
 w^j=0, &\text{on}\ \partial B_{3r/5}^{m}(0).
 \end{cases}
\end{equation*}
By \eqref{control-of-Laplace-of-zj},
\begin{equation}
\left\|\Delta_{h} z^j\right\|_{C^{0,\alpha}\left(B_{3r/5}^{m}(0)\right)} \leq \Psi(\tau\,|\, n,\alpha,Q),
\end{equation}
where the $C^{0,\alpha}$-norm is taken in the coordinates $\{z^j\}$.
By \eqref{C0-estimate-for-h}-\eqref{C1-estimate-for-h} and the Schauder estimates on Euclidean balls, there exists a constant $C(r)=C(n,\alpha,Q,r)>0$ such that
\begin{equation}\label{C2-close-of-yi-and-fi}
\begin{aligned}
\left\|w^j\right\|_{C^{2,\alpha}\left(B_{3r/5}^{m}(0)\right)}\leq C(r)\left\|\Delta_{h} z^j\right\|_{C^{0,\alpha}\left(B_{3r/5}^{m}(0)\right)}
\leq C(r)\Psi(\tau\,|\, n,\alpha, Q).
\end{aligned}
\end{equation}
By the definition of $C^{2,\alpha}$-norm, it yields
 \begin{equation}\label{C2-close-of-yi-and-fi-1}
 \begin{aligned}
 \left\|x^j-z^j\right\|_{C^0\left(B_{3r/5}^{m}(0)\right)}&+\sum_{k=1}^{m}\left\|\frac{\partial x^j}{\partial z^k}-\delta_{jk}\right\|_{C^0\left(B_{3r/5}^{m}(0)\right)}+\sum_{k,l=1}^{m}\left\|\frac{\partial^{2} x^j}{\partial z^k\partial z^l}\right\|_{C^0\left(B_{3r/5}^{m}(0)\right)}
 \\&+\sum_{k,l=1}^{m}\left\|\frac{\partial^{2} x^j}{\partial z^k\partial z^l}\right\|_{C^{\alpha}\left(B_{3r/5}^{m}(0)\right)}\leq C(r)\Psi(\tau\,|\,n,\alpha,Q).
 \end{aligned}
 \end{equation}

By taking $\tau$ sufficiently small,  (\ref{C2-close-of-yi-and-fi-1}) implies that (\ref{desired-harmonic-coordinate-chart})
 is a harmonic coordinate chart.
 \par
What remains is to verify (\ref{C1-norm-for-metric-h}).  Let $B=(\frac{\partial x^j}{\partial z^k})$ denote the coordinate translation matrix. Then $B^{-1}=(\frac{\partial z^{k}}{\partial x^j})$ is the inverse matrix of $B$. Therefore
\begin{equation}\label{represent-of-inverse-matrix}
\frac{\partial z^{k}}{\partial x^j}=\left(\text{polynomial in}\left\{\frac{\partial x^s}{\partial z^t}\right\}\right)/\operatorname{det}(B),
\end{equation}
where $\operatorname{det}(B)$ is the determinant of matrix $B$.
Note that
\begin{equation}\label{metric-components-in-coordinates-translation}
\begin{aligned}
\bar{h}_{ij}=h\left(\frac{\partial}{\partial x^{i}},\frac{\partial}{\partial x^{j}}\right)=h\left(\frac{\partial z^{k}}{\partial x^i}\frac{\partial}{\partial z^{k}},\frac{\partial z^{l}}{\partial x^j}\frac{\partial}{\partial z^{l}}\right)=\frac{\partial z^{k}}{\partial x^i}\frac{\partial z^{l}}{\partial x^j}h_{kl}.
\end{aligned}
\end{equation}
Thus, (\ref{C0-estimate-for-h})-(\ref{C1-estimate-for-h}) together with (\ref{C2-close-of-yi-and-fi-1})-(\ref{metric-components-in-coordinates-translation}) yield the $C^{1,\alpha}$-estimates (\ref{C1-norm-for-metric-h}) in the coordinates $\{x^j\}$.
\end{proof}

\section{$C^{1,\alpha}$-regularity on the limit space}

From this section we begin to prove Theorem \ref{thm-bounded-local-covering-geometry}. We first show that any limit space $X\in \mathcal X_{n,r,v}^m(\delta,\rho)$ is a $C^{1,\alpha}$-Riemannian manifold for any $0<\delta\le \delta(n)$ and $0< \alpha<1$.
Without loss of generality, we assume that $r=\rho$.

Let $(M_i,g_i)\overset{GH}{\longrightarrow} X$, where $(M_i,g_i)$ are Riemannian $n$-manifolds with $|\op{Ric}_{M_i}|\le n-1$ and $(\rho,v)$-local covering geometry. By Lemma \ref{lem-rewinding-reifenberg}, we assume directly that $(M_i,g_i)$ is of $(\delta(n), \rho)$-Reifenberg local covering geometry, where $\delta(n)$ is the constant in \eqref{def-Reifenberg-*}.

As already pointed out in the introduction, it suffices to consider the equivariant Gromov-Hausdorff convergence:
\begin{equation}\label{graph-local-cover}
\begin{CD}
(\widehat{B}(x_i,\frac{\rho}{2},\rho), \hat{x}_i, \Gamma_i) @>GH>i\to \infty> (Y, \hat x, G)\\
@V\pi_iVV @V\pi_\infty VV\\
(B_{\frac{\rho}{2}}(x_i,g_i),x_i) @>GH>i\to \infty> (B_{\frac{\rho}{2}}(x,X),x),
\end{CD}
\end{equation}
where the open balls and the normal covers are endowed with their length metrics, $\Gamma_i$ is the deck transformation by $\Gamma_{\frac{1}{2},\rho}(x_i)$, and $G$ is the limit group of $\Gamma_i$.

Let $\hat g_i=\pi_i^*g_i$ be the pullback Riemannian metric on the normal cover. By Remark \ref{rem-Reifenberg-normal-cover}, $\widehat{B}(x_i,\frac{\rho}{2},\rho)$ admits a uniform $C^{1,\alpha}$-harmonic radius at points definitely away from the boundary. Then by Anderson's convergence theorem \ref{thm-An-convergence}, the interior $Y^\circ$ of $Y$ is a smooth manifold, where $\hat g_i$ gives rise to a $C^{1,\alpha}$-Riemannian metric tensor $\hat h$ on $Y^\circ$.

Note that $Y/G$ is the limit of $(B_{\frac{\rho}{2}}(x_i,g_i),x_i)$ with its length metric. By the fact that $(B_{\frac{\rho}{4}}(x,Y/G),d_{Y/G})$ endowed with the restricted metric is isometric to $(B_{\frac{\rho}{4}}(x,X),d_X)$, and the same holds for any open ball centered at an interior point $z\in Y/G$ whose radius $< \frac{\rho}{4}-\frac{d(x,z)}{2}$, we derive that the interior of $Y/G$ is isometric to $B_{\frac{\rho}{2}}(x,X)$ equipped with the length metric.

\begin{proposition}\label{prop-C1alpha-regularity}
	There is $\delta=\delta(n)>0$ such that
	 \begin{enumerate}
	 	\item \label{prop-C1alpha-regularity-a} $\pi_\infty$ in \eqref{graph-local-cover} is a $C^\infty$-smooth submersion that is Riemannian between $C^{1,\alpha}$-Riemannian metrics.
	 	\item \label{prop-C1alpha-regularity-b} $X$ is a $C^\infty$-smooth manifold $(X,h)$ with a $C^{1,\alpha}$-Riemannian metric.
	 \end{enumerate}
\end{proposition}

We first verify that $X$ is regular, which is a direct corollary of the following observation.
\begin{lemma}\label{lem-Euclidean-quotient}
	Let $G$ be a closed subgroup of the isometry group of $\mathbb R^n$. Then any tangent cone of the quotient space $\mathbb R^n/G$ is either isometric to an Euclidean space, or definitely $\epsilon(n)$-away from any Euclidean space in the pointed Gromov-Hausdorff distance.
\end{lemma}
\begin{proof}
	Let $G(o)$ be the orbit of $G$ at the origin $o$ of $\mathbb R^n$, and $G_o\subset O(n)$ the isotropy group at $o$. Let $T_o^\perp=\mathbb R^m$ be the normal space of $G(o)$ at $o$, and $\bar o$ the quotient point of $o$ in $\mathbb R^n/G$. By the standard theory of isometric group actions (e.g. see \cite[Proposition 1.8]{Grove2000}), the tangent cone $T_{\bar o}$ at $\bar o\in \mathbb R^n/G$ is isometric to $T_o^\perp/G_o$. Hence it is an Euclidean space if and only if the action of $G_o$ on $T_o^\perp$ is trivial.
	
	In the following we assume $G_o$ acts on $T_o^\perp$ non-trivially.
	
	Case 1. the action of $G_o$ on $T^\perp_o$ is discrete. Then $T_o^\perp/G_o$ is an Euclidean cone $C(\Sigma)$ over $\Sigma=\mathbb S^{m-1}/G_o$, whose volume is no more than half of $\mathbb S^{m-1}$. Hence $T_{\bar o}$ is definitely away from $\mathbb R^m$.
	
	Case 2. the identity component of $G_o$ acts on $T^\perp_{o}$ non-trivially. Since the orbit of $G_o$ on $T^\perp_{o}$ must contain a great circle, the radius of $\mathbb S^{m-1}/G_o$ is no more than $\frac{\pi}{2}$. Hence $T_{\bar o}$ is also definitely away from any Euclidean space.
\end{proof}

\begin{lemma}\label{lem-regular-space}
	There is $\delta_1(n)>0$ such that for any $\delta\le \delta_1(n)$,
	if $X$ satisfies the $(\delta,\rho)$-Reifenberg condition, then it is regular.
\end{lemma}
\begin{proof}
	By definition \eqref{def-Reifenberg-*}, the unit ball in any tangent cone $T_x$ of $X$ is $\delta$-close to the unit Euclidean ball. At the same time, by \eqref{graph-local-cover} $T_x$ is a quotient space of $\mathbb R^n$ by an isometric group action. Lemma \ref{lem-regular-space} follows from Lemma \ref{lem-Euclidean-quotient} immediately.
\end{proof}

Next, we prove the key lemma in this section. Let us fix the pointed Gromov-Hausdorff approximation $\hat \alpha_i:(\widehat{B}(x_i,\frac{\rho}{2},\rho), \hat{x}_i, \Gamma_i) \to (Y, \hat x, G)$ in \eqref{graph-local-cover}.

\begin{lemma}\label{local-regularity-limit}
	For any $\delta<\delta_1(n)$, the action of $G$ on $Y$ in \eqref{graph-local-cover} is smooth and proper,
	such that the open ball $B_{\frac{\rho}{2}}(x,Y/G)$ is a $C^{1,\alpha}$-Riemannian manifold.
\end{lemma}
\begin{proof}
	Let us apply the Ricci flow on $(M_i,g_i)$. By Theorem \ref{thm-smoothing-ricci-flow}, the solution $g_i(t)$ of Ricci flow (\ref{ricci-flow-equation}) with initial condion $g_i(0)=g_i$
	exists for $0<t\leq T(n,\rho)$, a positive constant depending only on $n$ and $\rho$, such that \eqref{ineq-smoothing-sec} holds for $g_i(t)$.
	\par
	For any fixed $t\in(0,T(n,\rho)]$, let us view $\hat g_i(t)$ as a sequence of metrics on $\widehat{B}(x_i,\frac{\rho}{2},\rho)$ in \eqref{graph-local-cover}. Then by the Arzela-Ascoli Theorem under Gromov-Hausdorff convergence, the underlying distance functions $\hat d_{i,t}$ subconverges to a distance function $\hat d_{\infty,t}$, which is $e^{2t}$-bi-Lipschitz to the underlying distance of $\hat h$. By (\ref{ineq-smoothing-sec}) and Anderson's $C^{1,\alpha}$-convergence theorem, up to a subsequence $\hat g_i(t)$ uniformly converges to a smooth metric tensor $\hat h(t)$ on $(Y,\hat h)$ with the same regularity as $g_i(t)$ in \eqref{ineq-smoothing-sec}. Thus, we have
	\begin{equation}\label{graph-local-cover-smoothed}
	\begin{CD}
	(\widehat{U}_i, \hat d_{i,t}, \hat{x}_i, \Gamma_i) @>\hat \alpha_i>i\to \infty> (Y_t, \hat h(t), \hat x, G_t)\\
	@V\pi_iVV @V\pi_{t,\infty} VV\\
	(B_{\frac{\rho}{2}}(x_i,g_i), d_{i,t}, x_i) @>GH>i\to \infty> Y_t/G_t,
	\end{CD}
	\end{equation}
	where $\widehat{U}_i=\widehat{B}(x_i,\frac{\rho}{2},\rho)$, $Y_t$ is a smooth Riemannian manifold, and $G_t$ is the limit group action of $\Gamma_i$. (Furthermore, if $\hat \alpha_i$ is replaced by a local diffeomorphism $\psi_i$ that realizes the $C^{1,\alpha}$-convergence of $\hat g_i$ to $\hat h$, then it can be seen that $\psi_i^*\hat g_i(t)$ $C^\infty$-converges to $\hat h(t)$. We do not need this fact here.)
	
	The key point here is that the action of $\Gamma_i$ on $(\widehat{U}_i,\hat g_i(t))$ and $(\widehat{U}_i, \hat g_i)$ is the same as deck transformation of the normal cover $\pi_i:(\widehat{U}_i, \hat x_i) \to (B_{\frac{\rho}{2}}(x_i,g_i), x_i)$. And at the same time, $\hat g_i(t)$ and $\hat g_i$, and hence $\hat h(t)$ and $\hat h$, are two metric tensors on the same smooth manifolds $\widehat{U}_i$ and $Y$ respectively, which by \eqref{ineq-smoothing-sec} are $e^{2t}$-bi-Lipschitz equivalent.
	
	It follows from the Arzela-Ascoli Theorem that the limit action $G_t$ of $\Gamma_i$ under $g_i(t)$ is also the same as $G$ under $g_i$. Since the action of $G_t$ is smooth and proper on $(Y,\hat h(t))$, so is $G$ on $(Y,\hat h)$. Thus, we derive that $\pi_{t,\infty}$ coincides with $\pi_\infty$, and  $Y/G$ can be identified to $Y_t/G_t$ as topological spaces.
	
	In order to show that $B_{\frac{\rho}{2}}(x,Y/G)$ is a $C^{1,\alpha}$-Riemannian manifold, we first prove that $Y_t/G_t$ is a smooth Riemannian manifold.
		
	Indeed, the tangent cone of $Y/G$ is the same as that in $X$, which by the $(\delta,\rho)$-Reifenberg condition, is $\delta$-close to be Euclidean on the unit ball at the vertex point. By the $e^{2t}$-bi-Lipschitz equivalence between $Y_t/G_t$ and $Y/G$, the tangent cone of $Y_t/G_t$ is $(e^{2t}-1+\delta)$-close to be Euclidean on the unit ball. By Lemma \ref{lem-Euclidean-quotient}, for  $e^{2t}-1+\delta<\delta_1(n)$, $Y_t/G_t$ is regular.
	
	Then by the standard theory of isometric actions on Riemannian manifolds (e.g., \cite{Alexandrino2015}), the slice representation of the isotropy group $G_p$ at any interior point $p\in Y_t$ is trivial, hence the orbit $G_t(p)$ is principle. By the principal orbit theorem (e.g. see \cite[\S 1]{Grove2000}), it follows that $Y_t/G_t$ is a smooth Riemannian manifold, and $\pi_{t,\infty}$ is a smooth Riemannian submersion.

	Let us endow $B_{\frac{\rho}{2}}(x, Y/G)$ with the smooth structure on $Y_t/G_t$. Because $\pi_\infty:(Y,\hat h)\to  B_{\frac{\rho}{2}}(x_\infty,Y/G)$ coincides with $\pi_{t,\infty}$, $\pi_\infty$ is also a submersion. By the fact that $G$ acts on $Y$ isometrically, the Riemannian metric tensor $\hat h$ on $Y$ induces a quotient Riemannian metric $h$ on $Y_t/G_t$.
	
	Furthermore, by the implicit function theorem, an adapted coordinate chart can always be constructed around preimages of a point in $B_{\frac{\rho}{2}}(x,Y/G)$. By the proof of Lemma \ref{almost-C1-harmonic-coordinate-chart-on-X}, the adapted coordinate charts on $Y$ descend to $C^\infty$-admissible local coordinate charts on $B_{\frac{\rho}{2}}(x,Y/G)$, where the quotient metric tensor $h$ is $C^{1,\alpha}$.
\end{proof}

\begin{proof}[Proof of Proposition \ref{prop-C1alpha-regularity}]
	~
	\par
	Let $\delta_1(n)$ be that in Lemma \ref{lem-regular-space} and $\delta=\frac{1}{2}\delta_1(n)$. Then \eqref{prop-C1alpha-regularity-a} has been proved in Lemma \ref{local-regularity-limit}. In the following we prove \eqref{prop-C1alpha-regularity-b}. That is, $X$ admits a $C^\infty$-smooth differentiable structure, such that the metric tensor $h$ induced locally from $\pi_\infty$ in \eqref{graph-local-cover} is $C^{1,\alpha}$.
	
	It suffices to show the local charts induced from $\pi_{\infty}$ in \eqref{graph-local-cover} are $C^\infty$-admissible with each other, where the metric tensors coincide with each other by pulling back.
	
	Indeed, let $B_{\frac{\rho}{2}}(x_{1,i},g_i)$ and $B_{\frac{\rho}{2}}(x_{2,i},g_i)$ be two open balls, whose intersection $W_i(\frac{\rho}{2})=B_{\frac{\rho}{2}}(x_{1,i},g_i)\cap B_{\frac{\rho}{2}}(x_{2,i},g_i)$ is non-empty. Then the identity map $\imath_{j,i}$ from $(W_i(\frac{\rho}{2}), d_{W_i(\frac{\rho}{2})})$ with its length metric to $(W_i(\frac{\rho}{2}), d_{B_{\frac{\rho}{2}}(x_{j,i},g_i)})$ with the restricted metric is $1$-Lipschitz and locally isometric. By the precompactness Theorem \ref{thm-precompactness-b}, the normal cover $\pi_{i,W}:\widehat{W}_i(\frac{\rho}{2},\rho)\to (W_i(\frac{\rho}{2}), d_{W_i(\frac{\rho}{2})})$  sub-converges to $\pi_{\infty,W}:\widehat{W}_\infty (\frac{\rho}{2},\rho)\to W_\infty(\frac{\rho}{2})$ as $i\to \infty$.
	And the identity map $\imath_{j,i}$ sub-converges to $\imath_{j,\infty}:W_\infty(\frac{\rho}{2})\to W_{j,\infty}(\frac{\rho}{2})\subset Y_j/G_j=\lim_{i\to \infty}B_{\frac{\rho}{2}}(x_{j,i},g_i)$, which is also locally isometric.

	Furthermore, by \eqref{intersection-normal-cover},
	$\pi_{j,i,W}:\widehat{W}_i(\frac{\rho}{2},\rho)\to \pi_{j,i}^{-1}(W_i(\frac{\rho}{2}))\subset \widehat{B}(x_{j,i},\frac{\rho}{2},\rho)$ is a normal cover for $j=1,2$. Let $\pi_{j,\infty}:Y_j\to Y_j/G_j$ be the limit submersion of $\pi_{j,i}:\widehat{B}(x_i,\frac{\rho}{2},\rho)\to B_{\frac{\rho}{2}}(x_i,g_i)$ as in \eqref{graph-local-cover}. Then we derive the following commutative diagram
	$$\xymatrix{
		& & \widehat{W}_\infty(\frac{\rho}{2},\rho) \ar[dl]_{\pi_{1,\infty,W}} \ar[dr]^{\pi_{2,\infty,W}} \ar[dd]_{\pi_{\infty,W}} &\\
		&\pi_{1,\infty}^{-1}(W_{1,\infty} (\frac{\rho}{2})) \ar[d]_{\pi_{1,\infty}} & & \pi_{2,\infty}^{-1}(W_{2,\infty}(\frac{\rho}{2})) \ar[d]_{\pi_{2,\infty}} \\
		&W_{1,\infty}(\frac{\rho}{2})  & W_\infty(\frac{\rho}{2})\ar[r]_{\imath_{2,\infty}}\ar[l]^{\imath_{1,\infty}} &W_{2,\infty}(\frac{\rho}{2})}$$
	where the maps $\imath_{j,\infty}$ are locally isometric.
	
	By the proof of Lemma \ref{local-regularity-limit}, $W_{\infty}(\frac{\rho}{2})$ is also a smooth manifold and $\pi_{\infty,W}$ is a smooth submersion. By the commutative diagram above, local coordinate charts on $W_{j,\infty}(\frac{\rho}{2})$ from $Y_j/G_j$ coincide with those on $W_\infty(\frac{\rho}{2})$  descending from $\widehat{W}_{\infty}(\frac{\rho}{2},\rho)$. Hence they are $C^\infty$-admissible with each other. Moreover, the metric tensors on $W_{j,\infty}(\frac{\rho}{2})$ are also the same up to the diffeomorphism $\imath_{j,\infty}$.
\end{proof}

\begin{remark}\label{rem-smoothed-time}
	Let $(M_i,g_i)\overset{GH}{\longrightarrow} X\in \mathcal{X}_{n,r,v}^m(\delta(n),\rho)$, where $(M_i,g_i)$ are Riemannian $n$-manifolds with $(r,v)$-local covering geometry, and $\delta=\delta(n)$ is the minimum of that in \eqref{def-Reifenberg-*} and in Proposition \ref{prop-C1alpha-regularity}.
	By Theorem \ref{thm-smoothing-ricci-flow} (Dai-Wei-Ye \cite{DWY1996}, cf. \cite{CRX2017}), the Ricci flow solution $g_i(t)$ with initial metric $g_{i}(0)=g_i$ admits a uniform positive existence time $T(n,r,v,\rho)$, and the higher-ordered regularities \eqref{ineq-smoothing-sec}.
	Assume that $(M_i,g_i(t))$ sub-converges to $X_t$.
	We point out that,
	by the proof of Lemma \ref{local-regularity-limit}, for not only $t$ satisfying $e^{2t}-1+\delta(n)<\delta_1(n)$, but also all $t\in (0,T(n,r,v,\rho)]$, $X_t$ is regular.
	Hence, by the proof of Proposition \ref{prop-C1alpha-regularity}, $X_t$ is a smooth Riemannian manifold for any $t\in (0,T(n,r,v,\rho)]$.
	
	Indeed, this can be seen from the fact that $Y_t$ shares the same isotropy group as $Y_{t_0}$ for $e^{2t_0}-1+\delta(n)<\delta_1(n)$, or an open and closed argument on $t_0=\sup\{t:\text{$X_t$ is regular}\}$, due to that $e^{2t}-1=\frac{\delta_1(n)}{2}$ implies that the tangent cone of $(X,h(t))$ is isometric to $\mathbb R^m$, and hence $t_0$ is extended further to cover $2t$.

\end{remark}
\section{Adapted harmonic coordinates on the local covers of balls}
Let the assumptions be as in the beginning of Section 4.
In this section we construct adapted harmonic coordinates on $\widehat{B}(x_i,\frac{\rho}{2},\rho)$ and on its limit space $(Y, \hat x)$ in the graph \eqref{graph-local-cover} respectively. Let $\hat h$ be the Riemannian metric tensor on $Y$ and $h$ its quotient metric tensor on $X$.
\begin{proposition}\label{prop-adapted-harmonic-coordinates}
There are constants $\delta(n)>0$ and $r_1(n, r,v, \rho,\alpha,Q)>0$ such that for $0<\delta\leq \delta(n)$ and $R\le \min\{\delta^{-1/2}r_1, \delta^{-1/4}\rho\}$,
there exists an adapted harmonic coordinate chart $(\hat{f}^{1},\dots,\hat{f}^{m},\hat f^{m+1},\dots,\hat f^n)$ defined on $B_{R}(\hat{x}_\infty,\hat{h}_{\delta}=\delta^{-1} \hat h)$ such that $\hat{f}^{k}$ ($k=1,\dots,m$) descends to a smooth function $f^{k}$ in $B_{R}(x_\infty,h_{\delta})$, and admits the following regularities:
\begin{equation}\label{C1-norm-control-for-metric-tilde-h}
\begin{cases}
e^{-Q}\delta_{kl}\leq \hat{h}_{\delta,kl}=\hat{h}_{\delta}(\frac{\partial}{\partial \hat{f}^{k}},\frac{\partial}{\partial \hat{f}^{l}})\leq e^{Q}\delta_{kl};\\
R^{1+\alpha}\left\|\frac{\partial \hat{h}_{\delta,kl}}{\partial \hat{f}^{j}}\right\|_{C^{0,\alpha}(B_{R}(\hat x,\hat h_\delta))}\leq \Psi(\delta\mid n, r,v, \rho,\alpha,Q)
\end{cases}
\end{equation}
and
\begin{equation}\label{Hessian-control-for-yj}
\|\operatorname{Hess}\hat{f}^j\|_{C^{0,\alpha}(B_{R}(\hat x, \hat h_\delta))}\leq \Psi(\delta\mid n, r,v, \rho,\alpha,Q), \ \ j=1,\dots,n,
\end{equation}
where the $C^{0,\alpha}$-norm is taken in the coordinates $\{\hat{f}^{j}\}$.

\end{proposition}

Assuming Proposition \ref{prop-adapted-harmonic-coordinates}, we first prove Theorem \ref{thm-bounded-local-covering-geometry}.

\begin{proof}[Proof of Theorem \ref{thm-bounded-local-covering-geometry}]
	~
	\par
	Let $\delta=\delta(n)$ be that in Proposition \ref{prop-adapted-harmonic-coordinates}, and $X\in \mathcal{X}_{n,r,v}^m(\delta,\rho)$ a compact Ricci-limit space of Riemannian $n$-manifolds $(M_i,g_i)$ with $|\Ric_{M_i}|\leq n-1$ and $(r,v)$-local covering geometry, such that any point $x\in X$ is $(\delta,\rho)$-Reifenberg.
	By Lemma \ref{lem-rewinding-reifenberg}, $(M_i,g_i)$ is of $(\delta,r')$-Reifenberg local covering geometry for $r'=r'(n,r,v,\rho)>0$.
	
	Let $r_1=r_1(n,r',\alpha,Q)$ be as in Proposition \ref{prop-adapted-harmonic-coordinates}, and let $r_0=\min\{r_1,\delta^{1/4}\rho\}$. By Proposition \ref{prop-C1alpha-regularity}, $X\in \mathcal X_{n,r,v}^m(\delta,\rho)$ is a $C^\infty$-smooth manifold with a $C^{1,\alpha}$-Riemannian metric tensor $h$. Proposition \ref{prop-adapted-harmonic-coordinates} together with Theorem \ref{tech-theorem-*} implies the harmonic radius of $(X,h)$ is no less than $r_0$. The first part of (\ref{thm-bounded-local-covering-geometry}.1) is complete.
		
	For (\ref{thm-bounded-local-covering-geometry}.2), let us assume that $(X,h)\in\mathcal X_{n,r,v}^m(\delta,\rho,D)$.
	Since $(X,h)$ is a Ricci-limit space, there are harmonic coordinate charts that covers $(X,h)$ whose number admits a uniform bound $N(n,r_0,D)$, as well as the multiplicity of their intersections. By a standard argument, e.g., \cite[Lemma 2.1]{Anderson1990}, or \cite[arguments below Theorem 0.2]{AnCh1992}, $X_{n,r,v}^m(\delta,\rho,D)$ is compact in the $C^{1,\alpha}$-topology.
	
	What remains is to show (\ref{thm-bounded-local-covering-geometry}.1.a-c).
	
	Let $(M_i,g_i)\overset{GH}{\longrightarrow} X$, where $(M_i,g_i)$ are Riemannian $n$-manifolds with $|\op{Ric}_{M_i}|\le n-1$ and $(\delta,r')$-Reifenberg local covering geometry.
	Let us consider the Ricci flow solution $g_i(t)$ with initial metric $g_{i}(0)=g_i$, which exists for $t\in (0,T(n,r')]$ by Theorem \ref{thm-smoothing-ricci-flow} (Dai-Wei-Ye \cite{DWY1996}, cf. \cite{CRX2017}) and satisfies
	the higher-ordered regularities \eqref{ineq-smoothing-sec}. By Remark \ref{rem-smoothed-time},  for all $t\in (0,T(n,r'))$ the limit space $X_t$ of $(M_i,g_i(t))$ is a smooth Riemannian manifold $(X_t,h(t))$, which is $e^{2t}$-bi-Lipschitz equivalent to $(X,h)$.
	
	By Remark \ref{rem-smoothed-time} again, there is $\rho_1(n)>0$ such that $(X_t,h(t))$ lies in $\mathcal X_{n,r,v}^m(\delta,\rho_1,D_1)$ for $D_1=e^{T(n,r')}D$. Hence it admits a uniform harmonic radius for all $t\in [0,T(n,r')]$. By the $C^{1,\alpha}$-precompactness, there is a diffeomorphism $\varphi_t$ from $X$ to $X_t$ such that the pullback metric $\varphi_t^*h(t)$ converges to $h$ in the $C^{1,\alpha}$-norm.
	
	Let $\hat g_i(t)=\pi_i^*g_i(t)$ be the pullback metric on the normal cover $\widehat{B}(x_i,\frac{r'}{2},r')$ of $B_{\frac{r'}{2}}(x_i,g_i)$ in the graph \eqref{graph-local-cover-smoothed}. Then the regularities \eqref{ineq-smoothing-sec} pass to the limit metric $\hat h(t)$ on $Y_t$. By the proof of Lemma \ref{almost-C1-harmonic-coordinate-chart-on-X}, the relations \eqref{pi-relation-metric-tensor}-\eqref{C1-estimate-using-Riemmanian-submersion} between the quotient metric $h(t)$ and $\hat h(t)$ imply that $h(t)$ satisfies the same regularities \eqref{ineq-smoothing-sec} up to a definite ratio depending on $n$ and $Q$.
\end{proof}

\subsection{Preparation}
Let us make some preparation for the proof of Proposition \ref{prop-adapted-harmonic-coordinates}.
\begin{lemma}\label{lem-standard-blow-up}
	Let $\varphi=(y^1,\dots,y^n):B_r(p)\subset (M,h)\to \Omega\subseteq \mathbb{R}^{n}$ be a harmonic coordinate chart at $p$ with $\varphi(p)=0$ such that
	\begin{equation}\label{c0-estimate-for-star-tilde-h-*}
	e^{-Q}\delta_{st}\leq h_{st}=h(\frac{\partial}{\partial y^{s}},\frac{\partial}{\partial y^{t}})\leq e^{Q}\delta_{st}
	\end{equation}
	and
	\begin{equation}\label{c1-estimate-for-star-tilde-h-*}
	r^{1+\alpha}\left\|\frac{\partial h_{st}}{\partial y^{j}} \right\|_{C^{0,\alpha}\left(\Omega\right)}\leq e^{Q},
	\end{equation}
	where the $C^{0,\alpha}$-norm is taken in the coordinates $\{y^{j}\}$.
	Then after blowing up $h$ by $\lambda^2$, where $\lambda\to \infty$, the harmonic coordinates $\varphi_\lambda=(\lambda y^j):B_{\lambda r}(p,h_\lambda)\to \lambda\Omega$ $C^{2,\alpha}$-converge to a Cartesian coordinate system $(x^{1},\dots,x^{n})$ in $\mathbb{R}^{n}$, and $h_{\lambda}=\lambda^2 h$ $C^{1,\alpha}$-converge to an Euclidean metric $g_{\mathbb{R}^{n}}$. That is, for $0<R\le \lambda^{1/2}r$
	\begin{equation}\label{harmonic-coordinate-charts-close-to-Euclideac-space-*}
	\left\|y_{\lambda}^j\circ \varphi_\lambda^{-1}-x^{j}\right\|_{C^{2,\alpha}\left(B_{R}^{n}(0)\right)}\leq \Psi(\lambda^{-1}\,|\,r,\alpha,Q,n)
	\end{equation}
	and
	\begin{equation}\label{metric-tensor-close-to-Euclideac-space-*}
	\left\|h_{\lambda,st}\circ \varphi_\lambda^{-1}-\delta_{st}\right\|_{C^{1,\alpha}\left(B_{R}^{n}(0)\right)}\leq \Psi(\lambda^{-1}\,|\,r,\alpha,Q,n),
	\end{equation}
	where $\Psi(\lambda^{-1}\,|\,r,\alpha,Q,n)\to 0$ as $\lambda^{-1}\to 0$ with the other variables fixed.
\end{lemma}
\begin{proof}
	This is an elementary fact. Let us identify $(B_r(p,h),h)$ and $(\Omega, h=(\varphi^{-1})^*h)$ with the pullback metric via the inverse of $\varphi$. Then each $y^{j}$ is a harmonic function on $(\Omega,h)$. By a linear transformation, we can assume that $h_{st}(0)=\delta_{st}$. Put $y_{\lambda}^{j}=\lambda y^{j}$ and $h_{\lambda,st}=h_{\lambda}(\frac{\partial}{\partial y_{\lambda}^{s}},\frac{\partial}{\partial y_{\lambda}^{t}})$. Then $h_{\lambda,st}=h_{st}$, and by the rescaling property of the $C^{1,\alpha}$-harmonic radius in Remark \ref{scaling-invariant-harmonic-radius},
	\begin{equation}\label{c1-estimate-for-star-tilde-h-delta-*}
	\lambda^{1+\alpha} \cdot r^{1+\alpha}\left\|\frac{\partial h_{\lambda,st}}{\partial y_{\lambda}^{j}} \right\|_{C^{0,\alpha}\left(\Omega\right)}\leq e^{Q},
	\end{equation}
	where the $C^{0,\alpha}$-norm is taken in coordinates $\{y_{\lambda}^{j}\}$.
	
	Therefore, by \eqref{c1-estimate-for-star-tilde-h-delta-*} and the standard Schauder interior estimates,  $(y_{\lambda}^{1},\dots,y_{\lambda}^{n}):(\Omega,h_{\lambda})\to \mathbb{R}^{n}$ $C^{2,\alpha}$-converge to a standard Cartesian coordinate system in $\mathbb R^n$, and the metric $h_{\lambda}$ $C^{1,\alpha}$-converge to the standard Euclidean metric.
\end{proof}
Let us assume that $(X,h)$ is $(\delta,\rho)$-Reifenberg. Let $x\in X$ and $x_i\in M_i$ that converges to $x$. Let us consider the Gromov-Hausdorff convergence associated to the normal cover $\pi_i:(\widehat{U}_i, \hat g_i, \hat x_i)\to (B_{\frac{\rho}{2}}(x_i,g_i),x_i)$ as in the diagram \eqref{graph-local-cover}.

By Lemma \ref{lem-rewinding-reifenberg} and \cite{Anderson1990}, there exists $\delta(n)>0$ such that for given any $0<\alpha<1$, $Q>0$ and $0\leq\delta\leq\delta(n)$, the $C^{1,\alpha}$-harmonic radius of $(B_{\frac{\rho}{4}}(\hat{x}_i,\hat g_i),\hat{g}_{i})$ are uniformly bounded below by some constant $r_{1}=r_{1}(n,r,v,\rho,\alpha,Q)>0$. Assume that $\hat x_i$ converges to $\hat x\in Y$. By the continuity of $C^{1,\alpha}$-harmonic radius in Proposition \ref{continuity-of-harmonic-radius}, so is $B_{\frac{\rho}{4}}(\hat{x},\hat{h})\subset Y$.

After blowing up by $\delta^{-1}$, it follows from the definition of $(\delta,\rho)$-Reifenberg and Lemma \ref{lem-standard-blow-up} that
\begin{equation}\label{commutative-diagram-for-limit-space1-*}
\begin{CD}
	(B_{\delta^{-1/2}\rho/2}(\hat {x},\hat h_\delta),\hat {h}_{\delta}) @>C^{1,\alpha}>\delta\to 0>  (\mathbb{R}^{n},g_{\mathbb{R}^{n}})\\
	@V\pi_{\delta}VV @VV \pi_0 V\\
	(B_{\delta^{-1/2}\rho/2}(x,h_{\delta}),h_{\delta}) @>GH>\delta\to 0 > (\mathbb{R}^{m},g_{\mathbb{R}^{m}})
\end{CD}
\end{equation}
where $\pi_\delta$ is the restriction of $\pi_\infty$ in \eqref{graph-local-cover} on the local balls, $h_\delta=\delta^{-1}h$, and $\hat h_\delta=\pi_\delta^*(h_\delta)$. After passing to a subsequence, as $\delta\to 0$ $\pi_{\delta}$
converge to a submetry $\pi_{0}$, which is the canonical projection from $\mathbb{R}^{n}$ to $\mathbb{R}^{m}$.

Up to a suitable diffeomorphism, we view $B_{r_1}(\hat x_i,\hat g_i)$ as a fixed domain in $\mathbb R^n$ with metric $\hat g_i$. Let $\varphi_i:B_{r_1}(\hat x_i,\hat g_{i})\to \mathbb R^n$ be a sequence of harmonic coordinates $C^{2,\alpha}$-converges to a limit harmonic coordinates $\varphi:B_{r_1}(\hat x, \hat h)\to \mathbb R^n$ as $i\to \infty$. Let $\hat g_{i,\delta}=\delta^{-1}\hat g_i$. By Lemma \ref{lem-standard-blow-up} again, each
\begin{equation}\label{special-harmonic-coordinates}
\varphi_{i,\delta}=\delta^{-1/2}\varphi_i:B_{\delta^{-1/2}r_1}(\hat x_i,\hat g_{i,\delta})\to \mathbb R^n
\end{equation}
is uniformly $C^{2,\alpha}$-close to a Cartesian coordinates $(\hat x_{i}^1,\dots,\hat x_{i}^n)$ on $\mathbb R^n$ as $\delta\to 0$. By passing to a subsequence of $\{i\}$, we may assume that $(\hat x_{i}^1,\dots,\hat x_{i}^n)$ converges to a Cartesian coordinates $(\hat x^1,\dots,\hat x^n)$ with the following regularities: For any fixed $0<R<\delta^{-1/2}r_1$,
\begin{equation}\label{harmonic-coordinate-charts-close-to-Euclideac-space}
\left\|\varphi_{i,\delta}^j-\hat{x}_i^{j}\right\|_{C^{2,\alpha}\left(B_{R}^{n}(0)\right)}\leq \Psi(\delta\,|\,r_1,\alpha,Q,n)
\end{equation}
and
\begin{equation}\label{metric-tensor-close-to-Euclideac-space}
\left\|\hat{g}_{i,\delta,st}^*-\delta_{st}\right\|_{C^{1,\alpha}\left(B_{R}^{n}(0)\right)}\leq \Psi(\delta\,|\,r_1,\alpha,Q,n),
\end{equation}
where $(\hat x_i^1,\dots, \hat x_i^n)$ is the limit Cartesian coordinates system of $\varphi_{i,\delta}$ as $\delta\to 0$, $\hat g_{i,\delta,st}^*$ is the metric matrix of $\hat g_{i,\delta}^*$ in coordinates $\varphi_{i,\delta}$, and the norm is taken in the Euclidean coordinates $\{\hat{x}_i^{j}\}$ .

Furthermore, up to composing an orthonormal transformation on $(\hat x^1,\dots,\hat x^n)$ (also on $\varphi_{i,\delta}$), we assume that the limit projection $\pi_0: \mathbb R^n\to \mathbb R^m$ in the diagram \eqref{commutative-diagram-for-limit-space1-*}
\begin{equation}\label{projected-cartesian-coordinates}
\pi_0(\hat x^1,\dots,\hat x^m,\dots,\hat x^n)=(\hat x^1,\dots, \hat x^m),
\end{equation}
gives rise to a Cartesian coordiates on $\mathbb R^m$.

\subsection{Construction of the adapted harmonic coordinates}\label{subsec-construction-adapted-coordinates}
Based on the preparation above, let us begin the construction of the adapted harmonic coordinates on $\widehat{U}_i$ and their limit $Y$.

By the definition of $(\delta,\rho)$-Reifenberg, for all sufficiently large $i$ we have the Gromov-Hausdorff distance $$d_{GH}(B_{\delta^{-1/2}\rho}(x_i,\delta^{-1}g_{i}),B_{\delta^{-1/2}\rho}^m(0))\le 2\delta^{1/2}.$$
Let $\alpha_i:B_{\delta^{-1/2}\rho}(x_i,\delta^{-1}g_{i})\to B_{\delta^{-1/2}\rho}^m(0)\subset \mathbb R^m$ be an $2\delta^{1/2}$-Gromov-Hausdorff approximation.
Let $\{e_{1},\dots,e_{m}\}$ be the orthonormal basis at the origin of $\mathbb{R}^{m}$ associated to the Cartesian coordinates $(\hat x^1,\dots, \hat x^m)$ in \eqref{projected-cartesian-coordinates}.

Let $g_{i,\delta}=\delta^{-1}g_i$. Then for each $j=1,\dots,m$, there exists $p_{i,j}\in B_{\delta^{-1/2}\rho}(x_{i},g_{i,\delta})$ pairwise $2\delta^{1/2}$-close to $\delta^{-1/2}\rho e_j$ for all large $i$.
Let
\begin{equation}\label{buseman-func-below}
b_{i}^{j}(\cdot)=d_{g_{i,\delta}}(p_{i}^{j},\cdot)-d_{g_{i,\delta}}(p_{i}^{j},x_{i}),
\end{equation}
where $d_{g_{i,\delta}}$ denotes the distance induced by $g_{i,\delta}$ on $M_i$.
Let $f_{i}^{j}$ be the solution of the following Dirichlet problem for $R\le \delta^{-1/4}\rho$,
\begin{equation}\label{construction-harmonic-base}
\begin{cases}
\Delta_{g_{i,\delta}} f_{i}^{j}=0, &\text{in}\ B_{4R}(x_{i},g_{i,\delta});\\
f_{i}^{j}=b_{i}^{j},&\text{on}\ \partial B_{4R}(x_{i},g_{i,\delta}),
\end{cases}
\end{equation}
where $\Delta_{g_{i,\delta}}$ is the Laplace-Beltrami operator associated with metric $g_{i,\delta}$. $\Ric(M_{i},g_{i,\delta})\geq -(n-1)\delta^{1/2}$,  $(f_{i}^1,\dots, f_i^{m}):B_R(x_i)\to \mathbb R^m$ forms a $\delta$-splitting map in the sense of Cheeger-Colding \cite{CC1996}.

By \cite[Theorem 6.68]{CC1996} (or \cite[Lemma 1.23]{Colding1997}, or the quantitative maximum principles \cite[\S 8]{Cheeger2001} together with Abresch-Gromoll's excess estimate \cite{Abresch-Gromoll}), there is $\delta(n)>0$ such that for any $0<\delta\le \delta(n)$ and $R\le \delta^{-1/4}\rho$, the following $C^{0}$-estimate holds for all sufficiently large $i$:
\begin{equation}\label{C0-close-to-distance-function-*}
\left|f_{i}^{j}-b_{i}^{j}\right|\leq \Psi(\delta\,|\,n,\rho),\quad \text{in} \  B_{2R}(p_{i},g_{i,\delta}).
\end{equation}
As
$\pi_{i}:(\widehat{U}_i,\hat{g}_{i,\delta})\to B_{\delta^{-1/2}\rho/2}(x_{i},g_{i,\delta})$
is locally isometric, $\hat f_i^j=f_i^j\circ \pi_i$ is also harmonic, i.e.,
\begin{equation}\label{harmonic-of-lifting-function}
\Delta_{\hat{g}_{i,\delta}}\hat {f}_{i}^{j}=0,\ \text{in} \ \pi_{i}^{-1}(B_{2R}(p_{i},g_{i,\delta})).
\end{equation}

In the following we show that $(\hat f_i^1,\dots, \hat f_i^m, \varphi_{i,\delta}^{m+1},\dots, \varphi_{i,\delta}^n)$ still forms a harmonic coordinates. Since $\left(B_{\delta^{-1/2}\rho/2}(\hat{x}_{i},\hat{g}_{i,\delta}),\hat{g}_{i,\delta}\right)\overset{C^{1,\alpha}}{\longrightarrow}(V,\hat{h}_{\delta})\subset (Y,\hat h_\delta=\delta^{-1}\hat h)$ as $i\to \infty$, there exists a sequence of diffeomorphisms
$\psi_{i}:(V,\hat{h}_{\delta})\to B_{\delta^{-1/2}\rho/4}(\hat{x}_{i},\hat{g}_{i,\delta}),$
with $\psi_{i}(\hat{x})=\hat{x}_i$, such that the pullback metrics $\hat{g}_{i,\delta}^{*}=\psi_{i}^{*}\hat{g}_{i,\delta}$ (sub-)converge to $\hat{h}_{\delta}$ in the $C^{1,\alpha}$-norm. From now on, in the following subsections let us identify $(B_{\delta^{-1/2}\rho/4}(\hat{x}_{i},\hat{g}_{i,\delta}),\hat{g}_{i,\delta})$ with $(V, \hat g_{i,\delta}^*)$ via $\psi_i$.

\begin{lemma}\label{lem-adapted-coordinates}
	 The map $(-\hat f_i^1,\dots, -\hat f_i^m, \varphi_{i,\delta}^{m+1},\dots, \varphi_{i,\delta}^n):B_R(\hat x,\hat g_{i,\delta}^*)\to \mathbb R^n$ still forms a harmonic coordinates that converges to the Cartesian coordinates $(\hat x^1,\dots, \hat x^n)$ given by \eqref{projected-cartesian-coordinates} in the $C^{2,\alpha}$-norm as $i\to \infty$ and $\delta \to 0$.
\end{lemma}
\begin{proof}
	It suffices to show that $\hat f_i^j$ is $C^{2,\alpha}$-close to $-\varphi_{i,\delta}^j$ for every $j=1,\dots, m$.
	
	By \eqref{harmonic-coordinate-charts-close-to-Euclideac-space},  $\varphi_{i,\delta}$ is close to the Cartesian coordinates $(\hat x_{i}^j)$. Moreover, by the assumption under \eqref{special-harmonic-coordinates}, $(\hat x_{i}^j)$ converges to the Cartesian coordinates $(\hat x^j)_{j=1}^n$, which projects to the Cartesian coordinates $(\hat x^j)_{j=1}^m$.
	
	At the same time, let $b_i^j$ be the Buseman-typed function defined in \eqref{buseman-func-below} in the construction of $f_i^j$. Then by the choice of $p_i^j$, for any fixed $R>0$,
	\begin{equation}
	\left\|b_i^j+\hat x^j\circ \alpha_i\right\|_{C^0(B_R(x_i,g_{i,\delta}))}\le \Psi(i^{-1}, \delta\,|\, R, \rho)
	\end{equation}
	where $\alpha_i:B_{\delta^{-1/2}\rho}(x_i,g_{i,\delta})\to B_{\delta^{-1/2}\rho}^m(0)\subset \mathbb R^m$ is the Gromov-Hausdorff approximation.
	
	By the construction above, $\varphi_{i,\delta}$ converges to $(\hat x^1,\dots,\hat x^n)$ as $i\to \infty$ and $\delta\to 0$. Moreover, by \eqref{C0-close-to-distance-function-*} the difference between $f_i^j$ and $b_i^j$ goes to zero as $\delta\to0$ and $i\to \infty$. Let $(x^1,\dots,x^m)$ be the Cartesian coordinates on $\mathbb R^m$ in \eqref{commutative-diagram-for-limit-space1-*} such that $\hat x^j=x^j\circ \pi_0$. Then, by \eqref{projected-cartesian-coordinates} and the triangle inequality below
	$$ \|\hat f_i^j+\varphi_{i,\delta}^j\|\le \|f_i^j\circ \pi_i -b_i^j\circ\pi_i\|+\|b_i^j\circ \pi_i+ x^j\circ\alpha_i\circ  \pi_i\|+\|-x^j\circ \alpha_i\circ\pi_i+\hat x^j\|+\|-\hat x^j+\varphi_{i,\delta}^j\|,$$
	we derive
	\begin{equation*}
	\left\|\varphi_{i,\delta}^j+\hat{f}_i^{j}\right\|_{C^{0}\left(B_{2R}(\hat{x},\hat{h}_{\delta})\right)}\leq \Psi(i^{-1},\delta\,|\,n,r,v,\rho,\alpha,Q), \ \text{for each}\ j=1,\dots,m.
	\end{equation*}
	
	Since $\varphi_{i,\delta}^j+\hat{f}_i^{j}$ is also harmonic, by the Schauder interior estimates,
	\begin{equation}\label{C2-close-to-hat-fj}
	\begin{aligned}
	\left\|\varphi_{i,\delta}^j+\hat{f}_i^{j}\right\|_{C^{2,\alpha}\left(B_{R}(\hat{x},\hat{g}_{i,\delta}^*)\right)}&\leq C_{5}(n,r,v,\rho,\alpha,Q)\left\|\varphi_{i,\delta}^j+\hat{f}_i^{j}\right\|_{C^{0}\left(B_{R}(\hat{x},\hat{g}_{i,\delta}^*)\right)}
	\\&\leq \Psi(i^{-1},\delta\,|\,n,r,v,\rho,\alpha,Q),
	\end{aligned}
	\end{equation}
	where the $C^{2,\alpha}$-norm is taken in coordinates $\varphi_{i,\delta}$.

	Now together with \eqref{harmonic-coordinate-charts-close-to-Euclideac-space}
	and \eqref{metric-tensor-close-to-Euclideac-space}, \eqref{C2-close-to-hat-fj} implies that
	\begin{equation}
	(\hat y_i^1,\dots, \hat y_i^n):=(-\hat f_i^1,\dots, -\hat f_i^m, \varphi_{i,\delta}^{m+1},\dots, \varphi_{i,\delta}^n):B_{R}^n(0,\hat g_{i,\delta}^*)\to \mathbb R^n
	\end{equation} still forms a harmonic coordinates that converges to the Cartesian coordinates $(\hat x^1,\dots, \hat x^n)$ in the $C^{2,\alpha}$-norm as $i\to \infty$ and $\delta \to 0$, which satisfies
	\begin{equation}\label{adapted-charts-close-to-Euclideac-space}
	\left\|\hat y_i^j-\hat{x}^{j}\right\|_{C^{2,\alpha}\left(B_{R}^{n}(0,\hat g_{i,\delta}^*)\right)}\leq \Psi(i^{-1}, \delta\,|\,n,r,v,\rho,\alpha,Q)
	\end{equation}
	and
	\begin{equation}\label{adapted-metric-tensor-close-to-Euclideac-space}
	\left\|\hat{g}_{i,\delta,st}^*-\delta_{st}\right\|_{C^{1,\alpha}\left(B_{R}^{n}(0,\hat g_{i,\delta}^*)\right)}\leq \Psi(i^{-1},\delta\,|\,n,r,v,\rho,\alpha,Q).
	\end{equation}
	
\end{proof}

Now we are ready to prove Proposition \ref{prop-adapted-harmonic-coordinates}.

\begin{proof}[Proof of Proposition \ref{prop-adapted-harmonic-coordinates}]
	~
	\par
	Let $(\hat y_i^1,\dots, \hat y_i^n)=(\hat f_i^1,\dots, \hat f_i^m, \varphi_{i,\delta}^{m+1},\dots, \varphi_{i,\delta}^n):B_R(\hat x,\hat g_{i,\delta}^*)\to \mathbb R^n$  be the harmonic coordinate chart constructed in Lemma \ref{lem-adapted-coordinates}.
	Since $\hat g_{i,\delta}$ $C^{1,\alpha}$-converges to $\hat h$ as $i\to \infty$, by \eqref{adapted-charts-close-to-Euclideac-space}, $(\hat y_i^1,\dots, \hat y_i^n)$ $C^{2,\alpha}$-converges to a limit harmonic coordinates $(\hat y^1,\dots, \hat y^n):B_{R}(\hat x,\hat h_\delta)\to \mathbb R^n$ as $i\to \infty$.
	
	\par
	For each $j=1,\dots,m$, by the construction of $\hat y_i^j=\hat f_i^j=f_i^j\circ \pi_i$, its limits $\hat{y}^{j}$ takes the same value along every $\pi_\infty$-fiber, and thus it naturally descends to a $C^{2,\alpha}$-smooth function $f^{j}$ on $B_{R}(x,h_{\delta})\subset \delta^{-1/2}X$.

	\par
	What remains is to verify (\ref{C1-norm-control-for-metric-tilde-h}) and (\ref{Hessian-control-for-yj}). First, by \eqref{adapted-metric-tensor-close-to-Euclideac-space}, it is clear that the first inequality in (\ref{C1-norm-control-for-metric-tilde-h}) holds. Note that together with (\ref{Hessian-control-for-yj}), the first inequality implies the second in (\ref{C1-norm-control-for-metric-tilde-h}). It suffices to verify (\ref{Hessian-control-for-yj}).
	\par
	Secondly, by the $C^{2,\alpha}$-convergence of coordinate functions and \eqref{adapted-charts-close-to-Euclideac-space}, it is clear that (\ref{Hessian-control-for-yj}) holds.
	
	Now the proof of Proposition \ref{prop-adapted-harmonic-coordinates} is complete.
\end{proof}

\section{Harmonic radius estimate in terms of the volume}
This section is devoted to the proof of Theorem \ref{thm-Reifenberg-local-covering-geometry}.

\begin{proof}[Proof of Theorem \ref{thm-Reifenberg-local-covering-geometry}]
	~
	\par
	Let $\delta=\delta(n)$ be the constant in Theorem \ref{thm-bounded-local-covering-geometry}, and $X\in \mathcal{Z}_{n,\delta,r}^m(\delta)$.
	For \eqref{local-covering-geometry-1}, let us first prove that $X$ is a $C^{1,\alpha}$-Riemannian manifold with a positive harmonic radius.
	
	Indeed, let $(M_i,g_i)\overset{GH}{\longrightarrow}X$, where $(M_i,g_i)$ have $|\op{Ric}_{M_i}|\le (n-1)$ and $(\delta,r)$-Reifenberg local covering geometry. For any $x\in X$, let us consider the equivariant convergence of normal covers of $\frac{r}{2}$-balls in \eqref{graph-local-cover}.  Since every tangent cone $T_x$ at $x$ is $\delta$-close to $\mathbb R^m$, and the proofs of Propositions \ref{prop-C1alpha-regularity} and  \ref{prop-adapted-harmonic-coordinates} still go through for $X$. That is, $X$ is a smooth manifold and for any $x\in X$, the $\delta$-splitting map on $(M_i,g_i)$ defined by \eqref{construction-harmonic-base} from the Buseman functions by the closeness between $T_x$ and $\mathbb R^m$ gives rise to an adapted harmonic coordinate chart that descends from $Y$ in \eqref{graph-local-cover}, which satisfies the uniform regularities \eqref{C1-norm-control-for-metric-tilde-h} and \eqref{Hessian-control-for-yj}. Then by Theorem \ref{tech-theorem-*}, $X$ admits a $C^{1,\alpha}$-harmonic coordinate around any point $x\in X$. The continuity of $C^{1,\alpha}$-harmonic radius at points in $X$ yields a positive lower bound of the harmonic radius of $X$.

	
	Next, we show that the $C^{1,\alpha}$-harmonic radius at a point $x\in X$ satisfying
	\begin{equation}\label{local-volume-bound}
	\op{Vol}(B_R(x,X))\ge w>0
	\end{equation}
	admits a uniform bound $\ge r_0(n,r,w, R)>0$.
		
	Let us argue by contradiction. Assume that there is a sequence of spaces $(X_j,h_j)\in \mathcal{Z}_{n,\delta,r}^m(\delta)$, each of which contains a point $x_j\in X_j$ satisfying \eqref{local-volume-bound}, but the $C^{1,\alpha}$-harmonic radius $r_h(x_j,h_j)\to 0$. By passing to a subsequence we assume that $(X_j,x_j)\overset{GH}{\longrightarrow}(X,x)$.
	
	Let $(M_{j,i},g_{j,i})\overset{GH}{\longrightarrow}(X_j,h_j)$, where $(M_{j,i},g_{j,i})$ has bounded Ricci curvature and $(\delta,r)$-Reifenberg local covering geometry.
	By Theorem \ref{thm-smoothing-ricci-flow}, there is $t_0=t_0(n,r)>0$ such that any $(M_{j,i},g_{j,i})$ admits a smoothed metric $g_{j,i}(t_0)$ with uniformly higher regularities \eqref{ineq-smoothing-sec}. By passing to a diagonal subsequence, we assume that $$(M_{j,i},g_{j,i}(t_0))\overset{GH}{\longrightarrow}X_{j,t_0} \quad\text{as $i\to \infty$ for any fixed $j$, and }\quad X_{j,t_0}\overset{GH}{\longrightarrow}X_{t_0}.$$
	
	Since $(X_j,h_j)$ is regular and $X_{j,t_0}$ is $e^{2t_0}$-bi-Lipschitz to $(X_j,h_j)$, by Lemma \ref{lem-Euclidean-quotient}, each $X_{j,t_0}$ is also regular. Hence by Proposition \ref{prop-C1alpha-regularity} $X_{j,t_0}$ is a smooth Riemannian manifold $(X_j,h_j(t_0))$. Moreover, by O'Neill's formula applied on the Riemannian submersion $\pi_{t_0,\infty}$ in \eqref{graph-local-cover-smoothed}, the sectional curvature of $h_j(t_0)$ is bounded below uniformly by $C(n,r)t_0^{-1/2}$.
	
	Since the volume of $B_{2R}(x_j,h_j(t_0))$ is bounded below by $e^{-2mt_0}w$, their limit $X_{t_0}$ is an Alexandrov space with curvature $\ge C(n,r)t_0^{-1/2}$ of Hausdorff dimension $m$.
	
	By the proof of Theorem \ref{thm-limit-bounded-sec}, the sectional curvature of $(X_j,h_j(t_0))$ is also bounded uniformly from above. By \cite[Theorem 4.7]{CGT1982} the injectivity radius of $(X_j,h_j(t_0))$ is bounded below by $i_0(n,r,w,R)>0$.
	By Cheeger-Gromov's $C^{1,\alpha}$-precompactness, $X_{t_0}$ is regular. Then by the $e^{2t_0}$-bi-Lipschitz equivalence between $X_{t_0}$ and the original limit $X$ and Lemma \ref{lem-Euclidean-quotient} again, $X$ is also regular.
	
	Therefore, by the first part of \eqref{local-covering-geometry-1} $X$ is a $C^{1,\alpha}$-Riemannian manifold that admits a positive harmonic radius $r_\infty>0$. In particular, for any $p\in X$, $(X,\epsilon_i^{-1}h,p)$ is $\varkappa(\epsilon_i)$-close to $(T_pX,o)=(\mathbb R^m,0)$. Then $(X_j,\epsilon_i^{-1}h_j,x_j)$ is also $\varkappa(\epsilon_i)$-close to an Euclidean space $(\mathbb R^m,0)$ for fixed $i$ and any large $j$. Then harmonic coordinate charts of a definite radius can be constructed at $x_j\in (X_j,h_j)$ by the proof of Theorem \ref{thm-bounded-local-covering-geometry}. It contradicts to that $r_j\to 0$.
	
	Now Lemma \ref{lem-standard-blow-up} and Theorem \ref{thm-fibration-of-C1-manifolds} together imply \eqref{local-covering-geometry-2}.
\end{proof}

For limit spaces under bounded local covering geometry, we have the following $C^{1,\alpha}$-regularity that depends on the space itself.

Let $\mathcal{X}_{n,r,v}^m(\delta)$ be the set consisting of all compact Ricci-limit spaces of Riemannian $n$-manifolds with $|\Ric_M|\leq n-1$ and $(r,v)$-local covering geometry, such that each element $X\in \mathcal{X}_{n,r,v}^{m}(\delta)$ is $\delta$-almost regular.

\begin{corollary}\label{cor-reg-rewinding-volume} Let $(M_i,g_i)\overset{GH}{\longrightarrow} X\in \mathcal{X}_{n,r,v}^m(\delta)$, where $(M_i,g_i)$ are Riemannian $n$-manifolds with $|\op{Ric}_{M_i}|\le n-1$ and $(r,v)$-local covering geometry. Then the followings hold for $\delta=\delta_2(n)$.
	\begin{enumerate}
		\item\label{cor-reg-rewinding-volume-1}  For any sufficient large $i$ and any $x_i\in M_i$, the preimages of $x_i$ in the universal cover of $B_{r}(x_i)$ admit a uniform $C^{1,\alpha}$-harmonic radius $\ge r_0(X)>0$.
		\item\label{cor-reg-rewinding-volume-2} Any element $X\in \mathcal{X}_{n,r,v}^m(\delta)$ is regular in the sense that for any $x\in X$, any tangent cone $T_xX$ at $x$ is isometric to $\mathbb{R}^{m}$.
	\end{enumerate}
\end{corollary}

\begin{proof}
	Let $\delta_2(n)=\epsilon(n)/2$, where $\epsilon(n)$ is the constant that satisfies both Lemma \ref{lem-rewinding-reifenberg} and Lemma \ref{lem-Euclidean-quotient}.
	
	Suppose that $x_i\in M_i$ converges to $x\in X$, and $(s_jX, x)$ converges to a tangent cone $(T_x,o)$ that is $\delta_2(n)$-close to $\mathbb R^m$, where $s_j\to \infty$. Let us consider the equivariant pointed Gromov-Hausdorff convergence
	$$\begin{CD}
	(B_{s_j}(\tilde x_{i_j},s_j^2\tilde g_{i_j}), \tilde x_{i_j}, \Gamma_i) @>GH>> (Y, \tilde x, G)\\
	@V\pi_iVV @VV\pi V\\
	(B_{s_j}(x_{i_j},s_j^2g_{i_j}),x_{i_j}) @>GH>> (T_x,o)
	\end{CD}$$

	By Lemma \ref{lem-rewinding-reifenberg}, for any large $j$, $\tilde x_{i_j}\in B_{s_j}(\tilde x_{i_j},s_j^2\tilde g_{i_j})$ is a $(\delta(n),\rho(n,v))$-Reifenberg point, where $\delta(n)$ is the constant in \eqref{def-Reifenberg-*}. By $|\op{Ric}_{M_i}|\le n-1$ and the arguments in \cite{Anderson1990}, the $C^{1,\alpha}$-harmonic radius $r_h(\tilde x_{i_j},s_j^2\tilde g_{i_j})\ge r_0(n,v)>0$. By the $C^{1,\alpha}$-precompactness, $\tilde x$ is a regular point. Then by Lemma \ref{lem-Euclidean-quotient}, $T_x$ is isometric to $\mathbb R^m$. So we derive \eqref{cor-reg-rewinding-volume-2}.
	
	For \eqref{cor-reg-rewinding-volume-1}, let us consider the equivariant convergence of the normal covers of $B_{\frac{r}{2}}(x_i,g_i)$ in \eqref{graph-local-cover}. By Lemma \ref{lem-Euclidean-quotient}, every point in the limit space $Y/G$ of $B_{\frac{r}{2}}(x_i,g_i)$ is regular. Then by the proof of Lemma \ref{lem-rewinding-reifenberg}, the limit space of the normal cover $\widehat{B}(x_i,\frac{r}{2},r)$ is also regular, and thus $\hat x_i\in \widehat{B}(x_i,\frac{r}{2},r)$ is $(\delta(n),\rho(X))$-Reifenberg. Now   \eqref{cor-reg-rewinding-volume-1} follows.
\end{proof}

\section{Construction of canonical fibration under bounded Ricci curvature}
In this section, we prove Theorem \ref{thm-fibration-of-C1-manifolds}.

Let us first observe that Theorem \ref{thm-fibration-of-C1-manifolds} can be reduced to the case that $\rho=1$, where $\rho$ is the radius appearing in the Reifenberg condition and bounded local covering geometry.

Indeed, let $(M,g)$ be a closed Riemannian $n$-manifold with bounded Ricci curvature and $(\delta(n),\rho)$-Reifenberg local covering geometry, and $(X,h)$ a closed $C^{1,\alpha}$-Riemannian $m$-manifold in $\mathcal{Y}_{n}^{m}(\delta(n),\rho)$ for $0<\alpha<1$ and $0<\rho\le 1$, such that $$d_{GH}((M,g),(X,h))=\epsilon\cdot \rho.$$
If $\rho<1$, up to a rescaling $\rho^{-2}$ on the metric $g$ and $h$, $(M,\rho^{-2}g)$ is of $(\delta(n),1)$-Reifenberg local covering geometry and $(X,\rho^{-2}h)\in \mathcal X_n^m(\delta(n),1)$ with $$d_{GH}((M,\rho^{-2}g),(X,\rho^{-2}h))=\epsilon.$$ By viewing $g$ and $h$ as the rescaled metrics $\rho^{-2}g$ and $\rho^{-2}h$ respectively, we have $\rho=1$ and $$d_{GH}((M,g),(X,h))=\epsilon,$$ such that all rescaling invariant constants as in \eqref{thm-fukaya-1}-\eqref{thm-fukaya-3} remain the same.

We first prove a slightly weaker version of Theorem \ref{thm-fibration-of-C1-manifolds}. The center of mass technique below was applied earlier in \cite{NaberTian2016} and \cite{Huang2020} under the similar settings. For the difference see Remark \ref{rem-fibration-difference}.

\begin{proposition}\label{prop-C2alpha-fibration}
	Let $(M,g)$ be a closed Riemannian $n$-manifold with $|\op{Ric}_M|\le n-1$ and $(\delta(n),1)$-Reifenberg local covering geometry, and $(X,h)\in \mathcal{Y}_{n}^{m}(\delta(n),1)$.
	If $d_{GH}(M,X)\le \epsilon<\epsilon(n)$, then there is a $C^{2,\alpha}$-smooth fibration $f:M\to X$ that satisfies
	\begin{enumerate}
		\item \label{local-fibration-to-X1} $f:(M,g)\to (X,h)$ is a $\varkappa(\epsilon\,|\,n)$-almost Riemannian submersion,
		where after fixing $n$, $\varkappa(\epsilon\,|\,n)\to 0$ as $\epsilon\to 0$;
		\item \label{local-fibration-to-X2} the fiber $F_{y}=f^{-1}(y)$ has intrinsic diameter
		$\diam_{g}(F_{y}) \le C(n)\epsilon$ for any $y\in X$;
		\item \label{local-fibration-to-X3} the second fundamental form of $f$ satisfies
		$|\nabla^2f|_{g,h}\le C(n)$;
		\item \label{local-fibration-to-X4} $F_{y}$ is diffeomorphic to an infra-nilmanifold.
	\end{enumerate}
	
\end{proposition}
\begin{proof}[Proof of Proposition \ref{prop-C2alpha-fibration}]
	Let $\alpha:(M,g)\to (X,h)$ be an $\epsilon_1$-Gromov-Hausdorff approximation (for simplicity, $\epsilon_1$-GHA) with $\epsilon_1\le 2d_{GH}(M,X)$. Without loss of generality, we assume $\epsilon_1=\epsilon$.

	{\bf Step 1.} First, we prove that for $\tau=\tau(n,1)$, $0<\epsilon\le \tau^2$, and any $p\in M$, there is a local fibration $f_p:B_{\tau 100}(p,g)\to (X,h)$, which is a $\Psi(\epsilon\,|\,n)$-almost Riemannian submersion with a uniform second derivative control
	\begin{equation}\label{local-fibration-C2}
	|\nabla^2f_p|_{g,h}\le C(n),
	\end{equation}
	and $f_p$ satisfies
	\begin{equation}\label{local-fibration-close-GHA}
	\|f_p-\alpha\|_{C^0(B_{\tau 100}(p,g),(X,h))}\le \tau \cdot \Psi(\epsilon/\tau\,|\,n).
	\end{equation}

	Indeed, by the assumptions $(M,g)$ is of $(\delta(n),1)$-Reifenberg local covering geometry. Let us consider the equivariant closeness of the normal cover in the diagram \eqref{graph-local-cover-convergence}:
	\begin{equation}\label{graph-local-cover-convergence}
	\begin{CD}
	(\widehat{B}(p,\frac{1}{2},1), \hat g, \hat{p}, \Gamma) @>C^{1,\alpha}>\epsilon-\text{close}> (Y, \hat h, \hat y, G)\\
	@V\pi VV @V\pi_\infty VV\\
	(B_{\frac{1}{2}}(p,g),p) @>\alpha_p>\epsilon-\text{close}> (B_{\frac{1}{2}}(y,h),y)\subset Y/G,
	\end{CD}
	\end{equation}
	where $\alpha_p$ is a $\Psi(\epsilon\,|\,n)$-GHA, whose restriction on $B_{\frac{1}{4}}(p,g)$ coincides with $\alpha$.
	
	As the arguments before Lemma \ref{lem-adapted-coordinates} in Section \ref{subsec-construction-adapted-coordinates}, in the following we identify $B_{\frac{1}{4}}(\hat{p},\hat{g})$ with an open domain $V\subset Y$ with the the pullback metric $\hat{g}^*=\psi^*\hat{g}$ via a diffeomorphism $\psi:(V,\hat{h})\subset Y\to B_{\frac{1}{4}}(\hat{p},\hat{g})$, such that  $\psi(\hat{y})=\hat{p}$ and $\hat{g}^*$ is  $C^{1,\alpha}$-close to $\hat{h}$.
	
	Let $h_{\tau} = {\tau}^{-2}h$ and $g_\tau=\tau^{-2}g$. Since by Theorem \ref{thm-bounded-local-covering-geometry} the $C^{1,\alpha}$-harmonic radius of $(X,h)$ is no less than $r_0=r_0(n,1,\alpha,Q)>0$,  the rescaled $(X,h_\tau)$ has $C^{1,\alpha}$-harmonic radius at least $\tau^{-1}r_0$. By Lemma \ref{lem-standard-blow-up}, $B_{\tau^{-1/8}r_0}(\alpha(p),h_\tau)$ is $C(r_0,Q,n)\tau^{1/8}$-close to an Euclidean ball $B^m_{\tau^{-1/8}r_0}(0)\subset R^m$.
	
	Let us consider the harmonic $\delta$-splitting map constructed in \eqref{construction-harmonic-base},
	\begin{equation}\label{local-delta-splitting-map-*}
	u_{p,\tau}=(u^1,\dots,u^m):B_{200}(p,g_\tau)\to \mathbb R^m.
	\end{equation}
	Since the Ricci curvature of $g_\tau$ satisfies $\op{Ric}_{g_\tau}\ge -(n-1)\tau^2$. By the triangle inequality, $B_{\tau^{-1/8}r_0}(p,g_\tau)$ is $(C(r_0,Q,n)\tau^{1/8}+\epsilon/\tau)$-close to $B^m_{\tau^{-1/8}r_0}(0)$.
	Then there is $\tau=\tau(r_0,Q,n)$ such that for $0<\epsilon\le \tau^2$, \eqref{C0-close-to-distance-function-*} holds for $u^j$ and the Buseman functions $b^j$ as in \eqref{construction-harmonic-base}.
	
	In the following we fix $\tau=\tau(r_0,Q,n)$. By the proof of Proposition \ref{prop-adapted-harmonic-coordinates} applied on \eqref{graph-local-cover-convergence},
	the lifted harmonic functions $\hat u_{p,\tau}$ on $\widehat{B}(p,1/2,1)$ together with other harmonic functions form a $C^{1,\alpha}$-harmonic coordinate chart $\widehat{H}_{\hat p, \tau}:B_{100}(\hat p,\hat g^*_\tau)\to \mathbb R^n$, whose coordinate functions admit a uniform $C^{2,\alpha}$-norm bound $C(n)$.
	
	By the $C^{2,\alpha'}$-compactness for $0<\alpha'<\alpha$, $\widehat{H}_{\hat p, \tau}$ is $\Psi(\epsilon/\tau\,|\,n)$-$C^{2,\alpha'}$-close to an adapted harmonic coordinate chart $\widehat{H}_{\infty,\tau}:B_{100}(\hat p, \hat g^*_\tau)\to \mathbb R^n$ for $\pi_\infty$, which by Lemma \ref{almost-C1-harmonic-coordinate-chart-on-X} descends to an almost harmonic coordinate chart $H_{\infty,\tau}=(u_\infty^1,\dots,u_\infty^m)$ on $B_{100}(\alpha(p),h_{\tau})$, whose coordinate functions are $C^{2,\alpha'}$.
	
	By the almost commutative diagram \eqref{graph-local-cover-convergence},
	\begin{equation}\label{delta-splitting-C0-close}
	\left|u_{p,\tau}-H_{\infty,\tau}\circ \alpha\right|_{B_{100}(p,g_{\tau})}\le \Psi(\epsilon/\tau\,|\,n),
	\end{equation}
	and since $\hat u_{p,\tau}=u_{p,\tau}\circ\pi$,
	\begin{equation}\label{projection-C2-close}
	\left|u_{p,\tau}\circ \pi- H_{\infty,\tau}\circ \pi_\infty\right|_{C^{2,\alpha'}(B_{100}(\hat p,\hat g^*_\tau))}\le \Psi(\epsilon/\tau\,|\,n)
	\end{equation}
	
	Let us define $f_{p}=H_{\infty,\tau}^{-1}\circ u_{p,\tau}:B_{100}(p,g_{\tau})\to (X,h_\tau)$. Then \eqref{delta-splitting-C0-close} implies \eqref{local-fibration-close-GHA}.

	By \eqref{projection-C2-close}, $f_{p}\circ \pi:(V,\hat g^*)\to X$ is $C^{2,\alpha'}$-close to $\pi_\infty:(V,\hat h)\to X$, which is a Riemannian submersion. Hence $f_{p}$ is an $\varkappa(\epsilon\,|\,n)$-Riemannian submersion.

	By $\left|\nabla^2\hat u_{p,\tau}\right|_{B_{100}(\hat p,\hat g^*_\tau)}=\left|\nabla^2 u_{p,\tau}\right|_{B_{100}( p,g_\tau)}$ and the uniform bound on $H_{\infty,\tau}$ up to the 2nd covariant derivative in Lemma \ref{almost-C1-harmonic-coordinate-chart-on-X},
	$f_p$ admits a uniformly bounded Hessian. After rescaling back, we derive \eqref{local-fibration-C2}.
		
	{\bf Step 2.} Secondly, in order to obtain local fibrations that can be glued together, all operations below are done with respect to the fixed  $\epsilon$-GHA $\alpha:(M,g)\to (X,h)$.
	
	Let us consider the nearby metric $h(t_0)$ provided by Theorem \ref{thm-bounded-local-covering-geometry} for fixed $t_0=\frac{1}{2}\min\{T(n,1),\ln 2\}$. For $0<\delta\le \sqrt{t_0}$, let $h_{\delta,t_0}=\delta^{-2}h(t_0)$ and $h_\delta=\delta^{-2}h$. 	
	By (\ref{thm-bounded-local-covering-geometry}.1) for $\rho= 1$,
	$|\sec(X, h_{\delta,t_0})|\leq  C_1(n) t_0^{-1/2}\delta^2\leq  C_1(n) \delta$, and
	the $C^{1,\alpha}$-harmonic radius of $(X,h_{\delta,t_0})$ is at least $\delta^{-1}r_0(n)>0$.
	It follows that the injectivity radius of $(X,h_{\delta,t_0})$ admits a lower bound $\delta^{-1/2}i_0(n)>0$. Hence the convexity radius of $(X,h_{\delta,t_0})$ is no less than $\frac{1}{2}\delta^{-1/2}i_0(n)$. Without loss of generality, we assume that $\delta^{-1/2}i_0(n)\ge 100$.
	
	In the following we view the rescaled metric $g_\delta=(\tau/\delta)^2g_{\tau}$ as a rescaling of $g_{\tau}$, where $\tau=\tau(n,1)$ is provided by Step 1.
	Let $\{p_{j}\}_{j=1}^{N}$ be a $1$-net in $(M,g_{\delta})$ with $N$ depends only on the dimension $n$ of $M$. For each $p_{j}$, let $f_{j}:B_{100}(p_{j},g_{\tau})\to (X,h_\tau)$ be the local fibrations constructed in Step 1. Since $h_{\delta,t_0}$ is $e^{2t_0}$-bi-Lipschitz to $h_\delta$, \eqref{local-fibration-close-GHA} implies that
	\begin{equation}\label{epsilon-delta-close-local-fibration}
	\|f_{j}-\alpha\|_{C^0(B_{100\tau/\delta}(p_{j},g_{\delta}),(X,h_{\delta,t_0}))}\le (\tau/\delta)\cdot \Psi(\epsilon/\tau\,|\,n)
	\end{equation}
	with respect to $h_{\delta,t_0}$.
	
	By taking $\delta=\tau\sqrt{\Psi(\epsilon/\tau\,|\,n)}\to 0$ such that $(\tau/\delta)\cdot \Psi(\epsilon/\tau\,|\,n)\to 0$, we will glue such local fibrations $f_{j}$ together with respect to the smoothed metric $h_{\delta,t_0}$ to a $\varkappa(\epsilon,\delta\,|\,n)$-Riemannian submersion $f:(M,g_{\delta})\to (X,h_{\delta})$ with respect to the rescaled original metric $h_{\delta}$.
	
	Let $\phi:\mathbb R^1\to \mathbb R^1$ be a smooth cut-off function such that $\phi|_{[0,10]}\equiv 1$,  $\phi|_{[20,\infty)}\equiv 0$, and $|\phi'|,|\phi''|\le 10$.
	Let $\phi_{j}(x)=\phi (d_{h_{\delta,t_0}}(f_{j}(p_{j}),f_{j}(x)))$.
	Let us consider the energy function
	$$E:M\times X\to \mathbb R, \qquad E(x,y)=\frac{1}{2}\sum_j\phi_{j}(x)d_{h_{\delta,t_0}}(f_{j}(x),y)^2.$$
	By the construction of $\phi_{j}(x)$ and \eqref{epsilon-delta-close-local-fibration}, for $\Psi(\epsilon/\tau\,|\,n)<1$, $E(x,\cdot)$ is a strictly convex function in $B_{10}(\alpha_{\delta}(x), h_{\delta,t_0})$, and it takes a unique minimum point, $cm(x)$, that is $\sqrt{\Psi(\epsilon/\tau)}$-close to $\alpha_{\delta}(x)$ measured in $h_{\delta,t_0}$. We define
	\begin{equation*}
	f:(M,g)\to (X,h), \qquad f(x)=cm(x)
	\end{equation*}
	
	{\bf Step 3.} What remains is to verify (\ref{local-fibration-to-X1})-(\ref{local-fibration-to-X4}).
	\par
	Let us prove \eqref{local-fibration-to-X1}-\eqref{local-fibration-to-X3} first. By its definition $f$ is determined by the equations
	$$F(x,y)=\partial_y E=\sum_j\phi_j(x)\cdot \nabla \left(\frac{1}{2}d(f_j(x),\cdot)^2\right)=0,$$
	where $\partial_yE$ is the gradient of $E$ with respect to $y$.
	Note that, in the normal coordinates at $y=f(x)$, $F(x,y)$ can be written in the form $$F(x,y)=\left(-\sum_j\phi_j(x)f_j^1(x),\dots,-\sum_j\phi_j(x)f_j^m(x)\right),$$
	where $f_j(x)=(f_j^1(x),\dots, f_j^m(x))$ is the position vector of $f_j(x)$ in the normal coordinates of $y$.
	Then the differential of $f$ is determined by
	$df=-(\partial_y F)^{-1}\circ \partial_x F$, where
	$$-\partial_xF=\sum_j\left[d\phi_j\cdot f_j(x)+\phi_j(x)df_j \right].$$
	
	By the sectional curvature bound $C(n,1)t_0^{-1/2}$ of $h(t_0)$ provided by (\ref{full-rank1}.b),
	$\nabla^2 \frac{1}{2}r_{f_j(x)}^2$ is $\Psi_1(\delta\,|\,n,t_0)$-close to the identity matrix $E$, which is the Hessian of squared Euclidean distance. At the same time, by the Bishop-Gromov's relative volume comparison, the count of $j$ with non-vanishing $\phi_{j}(x)$ can be chosen at most $C_2(n)$. Hence $(\partial_yF)^{-1}$ is also $\Psi_2(\delta\,|\,n,t_0)$-close to $\left(\sum_j \phi_j(x)\right)^{-1}\cdot E$. It follows that for all sufficient small $\delta$,
	\begin{equation}\label{fibration-first-diff}
	\left|df-(\partial_yF)^{-1}\sum_j\phi_j(x) df_j\right|\le 2 \left(\sum_j \phi_j(x)\right)^{-1}\left|\sum_j d\phi_j(x)\cdot f_j(x)\right|.
	\end{equation}
	
	Since by \eqref{epsilon-delta-close-local-fibration} and the choice of $\delta$, $f_{j}(x)$ is  $\sqrt{\Psi(\epsilon/\tau\,|\,n)}$-close to $f(x)$ whenever $f_{j}(x)$ is well-defined,
	\begin{equation}\label{C0-small-local-fibration}
	|f_{j}(x)|\le \sqrt{\Psi(\epsilon/\tau\,|\,n)},\quad \text{for $x\in B_{40}(p_{j},g_{\delta})$.}
	\end{equation}
	Combing \eqref{fibration-first-diff} and \eqref{epsilon-delta-close-local-fibration} together, we derive
	\begin{equation}\label{fibration-first-diff-error}
	\left|df-\sum_j\frac{\phi_j(x)}{\sum_k \phi_k(x)} df_j\right|\le C_3(n)\left(\Psi_2(\delta\,|\,n,t_0)+\sqrt{\Psi(\epsilon/\tau\,|\,n)}\right).
	\end{equation}
	
	In order to show that $f$ is $\varkappa(\epsilon\,|\,n)$-almost Riemannian submersion, it suffices to show that the local fibrations $f_j$ nearby are $C^1$-close to each other.
	
	Indeed, by the definition of $f_j$ in Step 1, $H_{j,\infty,\tau}\circ f_{j}=u_{p_j,\tau}$ is a harmonic map, which by \eqref{epsilon-delta-close-local-fibration} is $\sqrt{\Psi(\epsilon/\tau\,|\,n)}$-close to each other up to a transformation in the intersection of their domains. In particular, for $j_1\neq j_2$
	\begin{equation}
	\left|u_{p_{j_1},\tau}-H_{j_1,\infty,\tau}\circ H_{j_2,\infty,\tau}^{-1}\circ u_{p_{j_2},\tau}\right|_{B_{40}(p_{j_1},g_\delta)\cap B_{40}(p_{j_2},g_\delta)}\le \sqrt{\Psi(\epsilon/\tau\,|\,n)}.
	\end{equation}
	At the same time, by the construction of $H_{j,\infty,\tau}$ (see also Proposition \ref{prop-adapted-harmonic-coordinates} and Lemma \ref{almost-C1-harmonic-coordinate-chart-on-X}), $H_{j_1,\infty,\tau}\circ H_{j_2,\infty,\tau}^{-1}$ on $B_{40}(p_{j_1},g_{\delta})\cap B_{40}(p_{j_2},g_{\delta})$ is $C(r_0,Q,n)\delta$-$C^{1,\alpha}$ close to a constant isometric transformation $A_{j_1j_2}$ on $\mathbb R^m$.
	
	Now by Cheng-Yau's gradient estimate \cite{CY1975} for the component harmonic functions of $$u_{p_{j_1},\tau}-A_{j_1j_2}u_{p_{j_2},\tau}+2\max\left|u_{p_{j_1},\tau}-A_{j_1j_2}u_{p_{j_2},\tau}\right|_{B_{30}(x,g_\delta)\subset B_{40}(p_{j_1},g_\delta)\cap B_{40}(p_{j_2},g_\delta)}$$ in the context of uniform lower Ricci curvature bound, it follows that, for $\eta_{j_1}(f_{j_1}(x))\neq 0$ and $\eta_{j_2}(f_{j_2}(x))\neq 0$,
	\begin{equation}\label{C1-close-local-fibration}
	\left|df_{j_1}-df_{j_2}\right|(x)\le C_4(n)\sqrt{\Psi(\epsilon/\tau\,|\,n)}.
	\end{equation}
	(Note that the support of $\eta_{j}$ lies in $B_{20}(p_{j},g_{\delta})$.)
	
 	Since $f_{j}:B_{40}(p_{j},g_{\delta})\to (X,h_\delta)$ is a $\varkappa(\epsilon\,|\,n)$-Riemannian submersion,
 	by $C^{1}$-closeness \eqref{C1-close-local-fibration} for $f_{j}$, the average of $df_j$ in \eqref{fibration-first-diff-error} is a $\varkappa(\epsilon\,|\,n)$-Riemannian submersion. Then \eqref{fibration-first-diff-error} implies \eqref{local-fibration-to-X1}.

	The uniform bound on the 2nd fundamental form of $f$ in \eqref{local-fibration-to-X3} follows from further calculations on the 2nd derivatives of implicit function and the uniform bounds (\ref{thm-bounded-local-covering-geometry}.1.b-c) on $\op{Rm}$ and $\nabla \op{Rm}$ of $h(t_0)$  .
	
	Indeed, the 2nd fundamental form of $f$ with respect to the Euclidean metric $g_{y}$ in the normal coordinates from $T_yX$ can be expressed in matrix by
	\begin{equation}\label{2nd-derivative-expression-*}
	d^2f=-(\partial_yF)^{-1}\cdot \left(\partial_x^2F+\partial_y(\partial_xF)\cdot df\right)-(\partial_yF)^{-1}\cdot\left( \partial_x(\partial_yF)\cdot df+\partial_y^2F\cdot (df)^2\right),
	\end{equation}
	where $\partial_x^2F$ consists of the Hessian of each components of $F$ with respect to $x$ such that
	\begin{equation}\label{2nd-derivative-expression}
	\partial_x^2F=\sum_j\nabla^2 \phi_j(x) \cdot f_j(x)+2d\phi_j(x)\otimes df_j(x)+\phi_j(x)\cdot \nabla^2f_j(x).
	\end{equation}
	
	For the last term of \eqref{2nd-derivative-expression-*}, we note that
	$\partial_yF=\partial^2_yE$ is a combination of
	$\nabla^2 \frac{1}{2}r_{f_j(x)}^2$, which is $\Psi(\delta\,|\,n,t_0)$-close to the identity matrix $E$. So is its inverse $(\partial_yF)^{-1}$.
	For $\partial_x(\partial_yF)$ and $\beta=\frac{1}{2}r^2_{f_j(x)}$, we have
	\begin{align*}
	\nabla_{df_j(X)}\left(\nabla^2\beta(Y,Y)\right)&=df_j(X)\left(YY\beta\right)-df_j(X)\left((\nabla_YY)\beta\right)\\
	&=YY\left<df_j(X),\nabla\beta\right>-(\nabla_YY)df_j(X)\beta.
	\end{align*}
	Observe that, $\left<df_j(X),\nabla\beta\right>=X^k\frac{\partial f_j^t}{\partial x^k}y^sh_{ts}(t_0)$ in the normal coordinates $(y^1,\dots,y^m)$ at $f_j(x)$. It follows that
	$\partial_x(\partial_y^2E)$ admits a uniform upper norm bound $C(n,t_0)$ with respect to $g$ and $h(t_0)$. So are $\partial_y(\partial_xF)$ and $\partial_y^2F$.
	
	At the same time, by \eqref{C0-small-local-fibration} and \eqref{local-fibration-C2} for $\delta=\tau\sqrt{\Psi(\epsilon/\tau\,|\,n)}$, each term in \eqref{2nd-derivative-expression} is bounded by $\delta\cdot C(n,t_0)$. It follows that  $|\nabla^2f|_{g_\delta,h_{\delta,t_0}}\le C(n)\sqrt{\Psi(\epsilon/\tau\,|\,n)}$ with respect to the rescaled metric $g_\delta$ and $h_{\delta,t_0}$.
	
	After rescaling back, we derive the bound of 2nd fundamental form in \eqref{local-fibration-to-X3} with respect to $h(t_0)$. By the $C^{1,\alpha}$-compactness in Theorem \ref{thm-bounded-local-covering-geometry}, it can be seen that after replacing $h(t_0)$ with the original metric $h$, \eqref{local-fibration-to-X3} still holds.
	
	The inequality \eqref{local-fibration-to-X2} follows from the same argument in \cite[proof of (2.6.1) of the fibration theorem 2.6]{CFG1992}. Indeed, up to a rescaling, let us assume that both the harmonic radius of $(X,h)$ and injectivity radius of $(X,h(t_0))$ is $\ge 1$, where $e^{2t_0}\le 1/2$.
	Suppose for some $y\in (X,h)$, the intrinsic diameter $$\diam_{g}(F_{y})\le \mu \epsilon,$$
	where $d_{GH}(M,X)\le \epsilon$.
	By the second fundamental form bound in \eqref{local-fibration-to-X3}, the extrinsic diameter of other fibers over $B_{1/4}(y,h(t_0))$ is no less than $C_6(n)\mu \epsilon$. By \eqref{local-fibration-to-X1}, at least $C_7(n)\epsilon^{-m}\mu$ many of $\epsilon$-balls are required to cover $f^{-1}(B_{1/4}(y,h(t_0)))\subset f^{-1}(B_{1/2}(y,h))$. However, by the existence of $\epsilon$-GHA $\alpha:(M,g)\to (X,h)$, which is $\Psi(\epsilon\,|\,n)$-close to $f$, and the harmonic coordinates on $B_1(y,h)$, at most $C_7(n)\epsilon^{-m}$ such balls are required. Hence $\mu\le C_8(n)$.
	
	The same argument in \cite[Step 3 of the proof of Proposition 6.6]{NaberZhang2016} yields \eqref{local-fibration-to-X4}. Here we give a simple proof by the regularities \eqref{local-fibration-to-X2} and  \eqref{local-fibration-to-X3}.
	
	Recall that by Theorem \ref{thm-smoothing-ricci-flow}, $g$ can be smoothed to $g(t_0)$ by the Ricci flow \cite{DWY1996}, whose sectional curvature is uniformly bounded by $C(n)t_0^{-1/2}$. By the $C^{1,\alpha}$-compactness, lifted to the universal cover of $1$-balls on $M$, the $C^{1,\alpha}$-norm between $\tilde g(t_0)$ and $\tilde g$ is uniformly bounded by $C(t_0,n)$. So is the Levi-Civita connections of $g(t_0)$ and $g$. By the expression of the second fundamental form in terms of Christoffel symbols, it is easy to see that $F_{y}=f^{-1}(y)$ with respect to $g(t_0)$ still satisfies \eqref{local-fibration-to-X3}. Hence, $F_{y}$ with the induced metric by $g(t_0)$ admits a uniform sectional curvature bound $C(t_0,n)$ with a small diameter $C(n)\epsilon$. The Gromov's almost flat theorem \cite{Gromov1978} (cf. also \cite{Ruh1982,Rong2019}) implies $(\ref{local-fibration-to-X4})$.
\end{proof}

\begin{proof}[Proof of Theorem \ref{thm-fibration-of-C1-manifolds}]
	~
	
The only difference from Proposition \ref{prop-C2alpha-fibration} is the $C^\infty$-smoothness of the fibration $f:M\to N$. Note that the $C^{2,\alpha'}$-regularity of $f$ in Proposition \ref{prop-C2alpha-fibration} is due to the limit coordinate chart $H_{\infty,\tau}$ on $B_{100}(\alpha(p),h_\tau)$ in the definition of $f_p$; see the paragraph below \eqref{projection-C2-close}.

Let $g(t)$ be the metric smoothed by Theorem \ref{thm-smoothing-ricci-flow}, and $h(t)$ be the metric on $X$ by the proof of Theorem \ref{thm-bounded-local-covering-geometry}. Let $H_{t,\infty,\tau}$ be the corresponding limit coordinate chart on $X$ with respect to $\tau^{-2}g(t)$ and $\tau^{-2}h(t)$. Then $H_{t,\infty,\tau}$ is $C^\infty$-smooth and $C^{2,\alpha}$-converges to $H_{\infty,\tau}$ as $t\to 0$.

By replacing $H_{\infty,\tau}$ by $H_{t,\infty,\tau}$ in the proof of Proposition \ref{prop-C2alpha-fibration}, we derive Theorem \ref{thm-fibration-of-C1-manifolds}.
\end{proof}

\begin{remark}\label{rem-local-almost-submersion}
	We point out that the local fibration $f_{p,\tau}$ with fixed $\tau=\tau(n,\rho)$ are modeled on the Riemannian submersion $\pi_\infty$ in \eqref{graph-local-cover-convergence} is necessary in guarantee that both \eqref{local-fibration-to-X1} and \eqref{local-fibration-to-X3} hold at the same time.
	
	If, instead of a rescaling of $f_{p,\tau}$, one picks up for each $\delta>0$ a local ``$\delta$-splitting map'' modeled on the Euclidean space as that in \eqref{construction-harmonic-base}, then by Cheeger-Colding's well-known $L^2$-estimates on their gradients and Hessians \cite{CC1997I,CC1996}, it also yields a $\Psi(\delta,\epsilon/\delta\,|\,n)$-almost Riemannian submersion $f_{p,\delta}:B_{100}(p,g_\delta)\to (X,h_\delta)$ for $0<\epsilon\le \delta^2(n)$ (cf. \cite[Proposition 6.6]{NaberZhang2016}), such that
	\begin{equation}\label{local-fibration-close-GHA-**}
	\|f_{p,\delta}-\alpha\|_{C^0(B_{100}(p,g_\delta),(X,h_\delta))}\le  \Psi(\delta,\epsilon/\delta\,|\,n),
	\end{equation}
	and
	\begin{equation}\label{local-fibration-C2-**}
	|\nabla^2f_{p,\delta}|_{g_\delta,h_\delta}\le C(n).
	\end{equation}
	
	But one immediately encounters the following issues:
	\begin{enumerate}
		\item to guarantee the local fibrations and the global fibration glued together are $\varkappa(\epsilon)$-almost Riemannian submersions, $\delta$ has to approach $0$;
		\item after the local fibrations are re-defined locally for each $\delta$, the 2nd derivative \eqref{local-fibration-C2-**} after rescaling back blows up as $\delta\to 0$.
	\end{enumerate}

	 Without a solution of the above, only a rescaling invariant 2nd derivative control on a $\varkappa(\epsilon\,|\,n)$-almost Riemannian submersion can be derive, such as
	 \begin{equation}\label{2nd-derivative-rescaling-invariant}
	 \left|\nabla^2f\right|_{g,h}(x)\cdot \diam_{g}(F_{f(x)})\leq C(n)\epsilon^{1/2}.
	 \end{equation}
\end{remark}

\begin{remark}\label{rem-fibration-smooth}
	There are several well-known methods (e.g., Hamilton's Ricci flow \cite{Hamilton1982} applied in \cite{DWY1996} or Perelman's pseudo-locality \cite{Perelman2002} (cf. \cite{HKRX2020}), embedding to Hilbert space by PDEs \cite{PWY1999} (cf. \cite{Abresch1988})),  by which a collapsed manifold $(M,g)$ with $|\op{Ric}_M|\le n-1$ and $(\delta,r)$-local covering geometry can be smoothed to a nearby metric $g_t$ that admits a uniform bounded sectional curvature depending on $t$. Hence by Theorem \ref{thm-fukaya} a fibration $f_t$ exists such that (\ref{thm-fukaya}.1-4) hold with respect to $g_t$.
	
	In order to remove the perturbing error that arises from $t$, such that the fibration $f_t:(M,g)\to (X,h)$ remains to be a $\varkappa(\epsilon)$-Gromov-Hausdorff approximation as $\epsilon\to 0$, $t$ must go to zero. Moreover, since the sectional curvature of $g_t$ generally blows up as $t\to 0$, some explicit curvature control on $g_t$ (e.g., $|\op{sec}_{g_t}|\le \frac{\alpha}{t}$ in \cite[Theorem 1.6]{HKRX2020} by Perelman's pseudo-locality \cite{Perelman2002}) and an arbitrary small distance distortion (e.g., $\le \alpha\sqrt{t}$ in \cite[Lemma 1.11]{HKRX2020}) on a definite scale are required. For details, see the proof of \cite[Theorem B]{HKRX2020}.
	
	Under the same setting of Theorem \ref{thm-fibration-of-C1-manifolds}, we construct fibrations in \cite{JKX2022} satisfying (\ref{thm-fukaya}.1-2) and (\ref{thm-fukaya}.4) via the smoothed metric $g(t)$ by the Ricci flow method and by suitably choosing the flow time with respect to $d_{GH}(M,X)$, among which \eqref{2nd-derivative-rescaling-invariant} holds as the best regularity for the 2nd order derivative.

	Though \eqref{2nd-derivative-rescaling-invariant} is strictly weaker than (\ref{thm-fukaya}.3), we prove in \cite{JKX2022} that all such fibrations are equivalent to each others as diffeomorphic types. Moreover, they are stable under Lipschitz perturbation on the metric $g$. It justifies the regularity \eqref{2nd-derivative-rescaling-invariant} is also suitable in describing the topology of full-rank collapsing phenomena under bounded Ricci curvature.
\end{remark}

	

\bibliographystyle{plain}
\bibliography{document}
\end{document}